\documentclass[preprint,authoryear,11pt]{elsarticle}
\usepackage{amsfonts,amssymb,amsmath,amsopn,amsthm}
\usepackage{graphics}
\usepackage{rotating}
\usepackage{bigints}
\usepackage{booktabs}
\usepackage{subfigure}
\usepackage{url}
\usepackage{amssymb}
\usepackage{color}
\usepackage{enumitem}
\usepackage[colorlinks=true]{hyperref}
\usepackage{setspace}
\newcommand{\ignore}[1]{}

\newtheorem{corollary}{{\sc Corollary}}
\newtheorem{thm}{\sc Theorem}

\newtheorem{pro}{\sc Proposition}
\newtheorem{lem}{\sc Lemma}
\newdefinition{rmk}{Remark}

\usepackage[margin=1in]{geometry}
\newproof{pf}{Proof}
\newproof{pot}{Proof of Theorem \ref{thm2}}

\DeclareMathOperator{\PRV}{PRV}
\DeclareMathOperator{\PCV}{PCV}
\DeclareMathOperator{\variance}{Var}

\allowdisplaybreaks

\makeatletter
    \renewcommand*{\subsection}{\@startsection{subsection}{2}{\z@}%
    {6pt}{6pt}{\reset@font\normalsize\bfseries}}
\makeatother

\begin{document}

\abovedisplayskip=6pt
\abovedisplayshortskip=6pt
\belowdisplayskip=6pt
\belowdisplayshortskip=6pt

\begin{frontmatter}
\title{Asymptotic properties of the realized skewness and related statistics}
\author[Tokyo,TMU,ISM,CREST]{Yuta Koike}
\ead{kyuta@ms.u-tokyo.ac.jp}
\author[UM]{Zhi Liu}
\ead{liuzhi@umac.mo}
\address[Tokyo]{Graduate School of Mathematical Sciences, University of Tokyo, Japan}
\address[TMU]{Department of Business Administration, Graduate School of Social Sciences, Tokyo Metropolitan University, Japan}
\address[ISM]{The Institute of Statistical Mathematics, 10-3 Midori-cho, Tachikawa, Tokyo 190-8562, Japan}
\address[CREST]{CREST, Japan Science and Technology Agency}
\address[UM]{Department of Mathematics, University of Macau}
\begin{abstract}

The recent empirical works have pointed out that the realized skewness, which is the sample skewness of intraday high-frequency returns of a financial asset, serves as forecasting future returns in the cross-section. Theoretically, the realized skewness is interpreted as the sample skewness of returns of a discretely observed semimartingale in a fixed interval. The aim of this paper is to investigate the asymptotic property of the realized skewness in such a framework. We also develop an estimation theory for the limiting characteristic of the realized skewness in a situation where measurement errors are present and sampling times are stochastic.
\end{abstract}
\end{frontmatter}
\newpage
\section{Introduction}  \label{sec1}

In the past decades, with widely available high frequency financial data, statistical inference for stochastic processes has significantly been developed. Among others, inference for the quadratic variation of a semimartingale using high frequency data is particularly of interest in the literature, due to its important applications in finance, namely, measuring the fluctuation of security markets; see \citet{JP1998}, \citet{Jacod2008}, \citet{ABM2005}, \citet{BR2006} and references therein.


In practice, the quadratic variation of a semimartingale is important in finance because it can be considered as a realized measure of the variance of short period returns. Besides, higher moments rather than the variance, in particular the third moment and the fourth moment which appear in measuring the skewness and kurtosis of assets, have attracted vast attention in finance, see \citet{Bakshi2003}, \citet{FW1980}, \citet{MZ2009}, \citet{HS1999,harvey2000}, \citet{MV2007}, \citet{KNS2013}, among others. By using high frequency data, the efficiency of estimating the quadratic variation has substantially improved. Thus, a natural question is whether we can achieve some improvements by using high frequency data in the inferences for higher order realized moments.
In the empirical aspect, recently \citet{ACJV2015} have showed strong evidence that the sample skewness of intraday high-frequency returns, which is called the \textit{realized skewness} in the paper\footnote{\cite{neuberger2012} uses the term realized skewness for a different concept.}, serves as predicting future equity returns in the cross-section. 
More precisely, they have found that if a stock's realized skewness averaged over a week is relatively higher (resp.~lower) than other stocks' ones (e.g.~more than the 90\% quantile (resp.~less than the 10\% quantile) of all stocks' ones), then the stock's return in the next week tends to be negative (resp.~positive). They have also confirmed that this empirical finding is robust across various implementations.  
The asymptotic property of the realized skewness is briefly discussed in \cite{ACJV2015} as well. The aim of this paper is to investigate this point more deeply. Specifically, suppose that the dynamics of the log price process of an asset is modeled by an It\^o semimartingale $X=(X_t)_{t\geq0}$ and we have discrete observation data $\{X_{i\Delta_n}\}_{i=0}^{\lfloor T/\Delta_n\rfloor}$ on the interval $[0,T]$, where $\Delta_n$ is a positive number tending to 0 as $n\to\infty$. Then the realized skewness is given by
\[
RDSkew=\frac{\lfloor T/\Delta_n\rfloor\sum_{i=1}^{\lfloor T/\Delta_n\rfloor}(\Delta_i^n X)^3}{\left\{\sum_{i=1}^{\lfloor T/\Delta_n\rfloor}(\Delta_i^nX)^2\right\}^{3/2}}.
\]
\cite{ACJV2015} have pointed out that
\[
\sum_{i=1}^{\lfloor T/\Delta_n\rfloor}(\Delta^n_iX)^2\to^P[X,X]_T\quad
\text{and}\quad
\sum_{i=1}^{\lfloor T/\Delta_n\rfloor}(\Delta^n_iX)^3\to^P\sum_{0\leq s\leq T}(\Delta X_s)^3,
\]
where $\Delta^n_iX=X_{i\Delta_n}-X_{(i-1)\Delta_n}$, $\to^P$ denotes convergence in probability, $[X,X]$ denotes the quadratic variation process of $X$ and $\Delta X_s=X_s-X_{s-}$. Hence the appropriately scaled realized skewness, $RDSkew/\lfloor T/\Delta_n\rfloor$, is a consistent estimator for the following quantity:
\begin{equation}\label{skewness}
\frac{\sum_{0\leq s\leq T}(\Delta X_s)^3}{\left([X,X]_T\right)^{3/2}}.
\end{equation}
In this paper we aim at deriving the asymptotic distribution of the estimation error
\begin{equation}\label{error}
\frac{1}{\lfloor T/\Delta_n\rfloor}RDSkew-\frac{\sum_{0\leq s\leq T}(\Delta X_s)^3}{\left([X,X]_T\right)^{3/2}}.
\end{equation}
We shall remark that the asymptotic property of the statistic of the form
\begin{equation}\label{functional}
\sum_{i=1}^{\lfloor T/\Delta_n\rfloor}g(\Delta^n_iX),
\end{equation}
where $g$ is a function on $\mathbb{R}$ satisfying some smoothness condition, is well-studied in the literature. To our knowledge, the most general condition to derive the asymptotic distribution of the above statistic is given by Theorem 5.1.2 from \cite{JP2012}, which requires that $g$ is of class $C^2$ and satisfies $g(0)=0$, $g'(0)=0$ and $g''(x)=o(|x|)$ as $x\to0$. Unfortunately, this condition is not satisfied by the cubic function $g(x)=x^3$, so this theorem is not applicable to deriving the asymptotic distribution of \eqref{error}. \citet{Kinnebrock2008} proved the result for $g(x)=x^3$ when $X$ is continuous. One aim of this paper is to fill in this gap.

Another important issue in high-frequency financial econometrics is to take account of microstructure noise and randomness of observation times: 
At ultra high-frequencies asset prices are usually modeled as discrete observations of a semimartingale with observation noise, which is referred to as microstructure noise, because such data typically exhibit several empirical properties which are inconsistent with the semimartingale assumption. 
In addition, ``raw'' high-frequency financial data are usually recorded at certain event times such as transaction times or order arrival times, which would be random and depend on observed values. 
See Chapters 7 and 9 of \cite{AJ2014} and references therein for more details on this topic. 
This paper also deals with this issue. Namely, we construct a consistent estimator for quantity \eqref{skewness} and develop an associated asymptotic distribution theory when microstructure noise is present and the sampling scheme is stochastic. 
To accomplish this, we study the asymptotic property of the ``pre-averaged'' version of the statistic \eqref{functional} with the cubic function $g(x)=x^3$. Here, ``pre-averaging'' is a de-noising scheme which enables us to systematically adapt functionals of semimartingale increments (such as \eqref{functional}) to the case that microstructure noise is present. The method was originally introduced in \cite{PV2009a} and subsequently generalized in \cite{JLMPV2009}, and many theoretical results on it are now available in the literature. In particular, under mild regularity assumptions, Theorem 16.3.1 of \cite{JP2012} provides the asymptotic distribution of the pre-averaged version of the statistic \eqref{functional} (in the equidistant case) when $g$ is a linear combination of positively homogeneous $C^2$ functions with degree (strictly) bigger than 3. Here, a function $f:\mathbb{R}\to\mathbb{R}$ is said to be positively homogeneous with degree $w\geq0$ if $f(\alpha x)=\alpha^wf(x)$ for any $\alpha\geq0$ and any $x\in\mathbb{R}$. Hence, the condition on the function $g$ again rules out the cubic function. We thus need to perform an additional analysis to cover the cubic function. 
We will also show that the randomness of observation times has no essential impact on the asymptotic distribution of the pre-averaged version of the statistic \eqref{functional} with the cubic function $g(x)=x^3$. This kind of phenomenon has already been observed in \cite{Koike2015jclt,Koike2015} for the pre-averaged version of the realized volatility. It contrasts the non-noisy case because the randomness of observation times can cause non-trivial modification of the asymptotic distribution of the realized volatility as illustrated in \cite{Fukasawa2010}, \cite{LMRZZ2014}, \cite{BV2014} and \cite{VZ2016} for example. 

The remainder of this paper is arranged as follows. Section \ref{sec2} investigates the asymptotic property of statistic \eqref{functional} and derives the asymptotic distribution of the realized skewness. Section \ref{section:noise} develops an estimation theory in a situation with microstructure noise and stochastic sampling times. Section \ref{section:proofs} is devoted to the proofs.

\section{The asymptotic distribution of the realized skewness}\label{sec2}

On a filtered probability space $\mathcal{B}=(\Omega, {\cal F}, (\mathcal{F}_t)_{t\geq 0}, P)$, we consider a stochastic process $(X_t)_{t\geq 0}$ of the form
\begin{equation*}
\textstyle
X_t=X_0+\int_0^tb_sds+\int_0^t\sigma_sdB_s+\int_0^t\int_{|\delta(s,z)|\leq 1}\delta(s,z)(\mu-\nu)(ds,dz)
+\int_0^t\int_{|\delta(s,z)|>1}\delta(s,z)\mu(ds, dz),
\end{equation*}
where the drift process $b$ is $(\mathcal{F}_t)$-progressively measurable, the spot volatility process $\sigma$ is $(\mathcal{F}_t)$-adapted and c$\grave{\mbox a}$dl$\grave{\mbox a}$g, $B$ is a standard Brownian motion, $\mu$ is a Poisson random measure on $\mathbb{R}^+\times E$ with predictable compensator $\nu(dt,dz)=dt\lambda(dz)$ and $\lambda$ being a $\sigma$-finite measure on a Polish space $(E,\mathcal{E})$, and $\delta$ is a predictable function on $\Omega\times\mathbb{R}^+\times E$.

We impose the following standard structural assumption:
\begin{enumerate}[noitemsep,label={\normalfont[H]}]

\item \label{asu2} There are a sequence $(\tau_k)$ of stopping times increasing to infinity and a sequence $(\gamma_k)$ of deterministic nonnegative measurable functions on $E$ such that $\int\gamma_k(z)^2\lambda(dz)<\infty$ and $|\delta(\omega,t,z)|\wedge1\leq\gamma_k(z)$ for all $k$ and all $(\omega,t,z)$ with $t\leq\tau_k(\omega)$.

\end{enumerate}


Let us assume that we observe the process $X$ at equidistant discrete points $\{i\Delta_n\}_{i=0}^{\lfloor T/\Delta_n\rfloor}$ for some $T>0$, where $\Delta_n$ is a sequence of positive numbers tending to zero as $n\to\infty$. We develop a central limit theorem for the non-normalized increments of $X$
\begin{equation*}
\mathcal{V}^n_T(X,g):=\sum_{i=1}^{\lfloor t/\Delta_n\rfloor}g(\Delta^n_iX)
\end{equation*}
for a function $g:\mathbb{R}\to\mathbb{R}$ satisfying some smoothness condition. If $g$ is continuous and satisfies $g(x)=o(x^2)$ as $x\to0$, it is known that
\[
\mathcal{V}^n_T(X,g)\to^P\sum_{0\leq s\leq T}g(\Delta X_s)
\]
as $n\to\infty$; see e.g.~Theorem 3.3.1 from \cite{JP2012}. If further $g$ is of class $C^2$ and satisfies $g(0)=g'(0)=0$ and $g''(x)=o(|x|)$ as $x\to0$, a central limit theorem for $\mathcal{V}^n_T(X,g)$ is also known under Assumption \ref{asu2} (see Theorem 5.1.2 of \cite{JP2012}). This condition is, however, not sufficient to allow the cubic function $g(x) = x^3$, which is crucial for deriving the asymptotic distribution of the realized skewness. Motivated by this reason, in the following we relax this condition to incorporate such a function.

We will use the notion of \textit{stable convergence} denoted by $\rightarrow^{\mathcal S}$. Here we briefly describe it before the main theorems. Let $(\mathcal{X},\mathcal{A},\mathbb{P})$ be a probability space and assume that we have a random element $Z_n$ taking values in a Polish space $S$ and defined on an extension $(\mathcal{X}_n,\mathcal{A}_n,\mathbb{P}_n)$ of $(\mathcal{X},\mathcal{A},\mathbb{P})$ for each $n\in\mathbb{N}\cup\{\infty\}$. In this setup the sequence $Z_n$ is said to \textit{converge stably in law} to $Z_\infty$ if $\mathbb{E}_n[Uf(Z_n)]\rightarrow \mathbb{E}_\infty[Uf(Z_\infty)]$ for any $\mathcal{A}$-measurable bounded random variable $U$ and any bounded continuous function $f$ on $S$. The most important property of this notion is the following: For each $n\in\mathbb{N}$, let $V_n$ be a real-valued variable on $(\mathcal{X}_n,\mathcal{A}_n,\mathbb{P}_n)$, and suppose that the sequence $V_n$ converges in probability to a variable $V$ on $(\mathcal{X},\mathcal{A},\mathbb{P})$. Then we have $(Z_n,V_n)\to^{d_s}(Z_\infty,V)$ for the product topology on the space $S\times\mathbb{R}$, provided that $Z_n\to^{\mathcal{S}}Z$. We refer to Section 2.2.1 of \citet{JP2012} for more detailed discussions.


We need some ingredients to describe the limiting random variables appearing in the central limit theorems below. Consider an auxiliary space $(\Omega',\mathcal{F}',P')$ supporting a  standard normal variable $U^0$, two sequences $(U_q)_{q\geq1}$, $(U'_q)_{q\geq1}$ of standard normal variables, and a sequence $(\kappa_q)_{q\geq1}$ of variables uniformly distributed on $(0,1)$, all of these being mutually independent. Then we introduce the extension $(\widetilde{\Omega},\widetilde{\mathcal{F}},\widetilde{P})$ of $(\Omega,\mathcal{F},P)$ by putting
$\widetilde{\Omega}=\Omega\times\Omega',$
$\widetilde{\mathcal{F}}=\mathcal{F}\otimes\mathcal{F}'$ and
$\widetilde{P}=P\times P'$.
Now let $(T_q)_{q\geq1}$ be a sequence of stopping times exhausting the jumps of $X$. Namely, $\{s\geq0:\Delta X_s(\omega)\neq0\}=\{T_q(\omega):q\geq1,T_q(\omega)<\infty\}$ for almost all $\omega$ and $T_q\neq T_{q'}$ if $q\neq q'$ and $T_q<\infty$. It is well-known that such a sequence always exists as long as $X$ is c\`adl\`ag and adapted; see Proposition I-1.32 of \cite{JS2003}. For any $C^1$ function $g:\mathbb{R}\to\mathbb{R}$ such that $g'(x)=o(|x|)$ as $x\to0$, we define the random variable $\overline{\mathcal{V}}_T(X,g)$ by
\begin{align*}
\overline{\mathcal{V}}_T(X,g)=\sum_{q:T_q\leq T}g'(\Delta X_{T_q})R_q,
\end{align*}
where $R_q=\sqrt{\kappa_q}\sigma_{T_q-}U_q+\sqrt{1-\kappa_q}\sigma_{T_q}U'_q.$ From Proposition 5.1.1 of \cite{JP2012} the variable $\overline{\mathcal{V}}_T(X,g)$ is well-defined and its $\mathcal{F}$-conditional law does not depend on the choice of the exhausting sequence $(T_q)$. For any integer $r\geq2$ we also define the random variable $\mathcal{Z}_T(X,r)$ by $\mathcal{Z}_T(X,r)=\overline{\mathcal{V}}_T(X,g_r),$ where the function $g_r$ is defined by $g_r(x)=x^r$.

In order to derive the asymptotic distribution of the realized skewness we also need to consider the \textit{realized volatility} of $X$, i.e.~$\text{RV}^n_T(X)=\sum_{i=1}^{\lfloor T/\Delta_n\rfloor}(\Delta^n_iX)^2$. Under Assumption \ref{asu2}, the following central limit theorem for $\text{RV}^n_T(X)$ is known (e.g.~Theorem 5.4.2 of \cite{JP2012}):
\[
\frac{1}{\sqrt{\Delta_n}}(\text{RV}^n_T(X)-[X,X]_T)\to^\mathcal{S}\sqrt{2IQ_T}U^0+\mathcal{Z}_T(X,2)
\]
as $n\to\infty$, where $IQ_T:=\int_0^T\sigma_s^4ds$ is the so-called \textit{integrated quarticity}. Here our aim is to develop a joint central limit theorem for the bivariate variables $(\text{RV}_T^n(X), \mathcal{V}_T^n(X, g))$. 
In the following the variables $\mathcal{Z}_T(X, 2)$ and $\overline{\mathcal{V}}_T(X, g)$ are defined with respect to the same auxiliary sequence $R_q$. 
\begin{thm}\label{thm2}
Let $g$ be a real-valued $C^2$ function on $\mathbb{R}$ satisfying $g(0)=g'(0)=0$ and $g''(x)=O(|x|)$ as $x\to0$. Under Assumption \ref{asu2}, the random variables
\[
\frac{1}{\sqrt{\Delta_n}}\left(\mathrm{RV}^n_T(X)-[X,X]_T,\mathcal{V}^n_T(X,g)-\mathcal{V}^n_T(X^c,g)-\sum_{0\leq s\leq T}g(\Delta X_s)\right)
\]
converge stably in law to
\[
(\sqrt{2IQ_T}U^0+\mathcal{Z}_T(X,2),\overline{\mathcal{V}}_T(X,g))
\]
as $n\to\infty$, where $X^c$ denotes the continuous martingale part of $X$, i.e.~$X^c_t=\int_0^t\sigma_sdB_s$.
\end{thm}
We prove this result in Section \ref{proof:thm2}.

\begin{rmk}

(i) If in addition $g''(x)=o(|x|)$ as $x\to0$, it can easily be seen that $\mathcal{V}^n_T(X^c,g)=o_P(\sqrt{\Delta_n})$ as $n\to\infty$, hence the theorem is a special case of Theorem 5.5.1 from \cite{JP2012} once we note that $[X,X]_{\lfloor T/\Delta_n\rfloor\Delta_n}-[X,X]_T=o_P(\sqrt{\Delta_n})$ under the assumptions of the theorem.

\noindent(ii) If the probability limit $\mathcal{V}_T(X^c,g):=P\text{-}\lim\mathcal{V}^n_T(X^c,g)/\sqrt{\Delta_n}$ exists, by using the properties of stable convergence we can deduce a central limit theorem for $\frac{1}{\sqrt{\Delta_n}}(\mathcal{V}^n_T(X,g)-\sum_{0\leq s\leq T}g(\Delta X_s))$ (in this case $\mathcal{V}_T(X^c,g)$ appears as the $\mathcal{F}$-conditional mean of the limiting variable).

\noindent(iii) If $g$ is positively homogeneous, i.e.~there exists a constant $w$ such that $g(\alpha x)=\alpha^wg(x)$ for any $\alpha\geq0$ and any $x\in\mathbb{R}$, the probability limit of $\mathcal{V}^n_T(X^c,g)/\sqrt{\Delta_n}$ can be derived from e.g.~Theorem 3.4.1 of \cite{JP2012}. In the following we give two examples of such a case as corollaries.

\noindent(iv) In general, the variables $\mathcal{V}^n_T(X^c,g)/\sqrt{\Delta_n}$ may not converge in probability (even in law, indeed); see the next proposition (we prove it in Section \ref{proof:counter-ex}).  

\end{rmk}

\begin{pro}\label{prop:counter-ex}
For every $a\in\mathbb{R}$, define the function $g_a:\mathbb{R}\to\mathbb{R}$ by $g_a(x)=|x|^3\sin(2a\log |x|)$.
\begin{enumerate}[label={\normalfont(\alph*)}]

\item For all $a\in\mathbb{R}$, $g_a$ is a $C^2$ function and satisfies $g_a(0)=g_a'(0)=0$ and $g_a''(x)=O(|x|)$ as $x\to0$. 

\item Suppose that $\sigma_s=1$ for all $s\in[0,T]$. There is a real number $a\neq0$ such that the variables $\mathcal{V}^n_T(X^c,g_a)/\sqrt{\Delta_n}$ do not converge in law with $\Delta_n=\exp(-n\pi/a)$. 

\end{enumerate}
\end{pro}

If we assume $g(x)=|x|^3$, we obtain a generalization of case 2 from Section 1.4.2 of \cite{JP2012} (p.20 of that book) where $X$ is assumed to be a scaled Brownian motion with a linear drift plus a compound Poisson process to a situation where $X$ is a more general It\^o semimartingale:
\begin{corollary}
Under Assumption \ref{asu2},
\begin{align*}
\frac{1}{\sqrt{\Delta_n}}\left(\sum_{i=1}^{\lfloor T/\Delta_n\rfloor}|\Delta^n_iX|^3-\sum_{0\leq s\leq T}|\Delta X_s|^3\right)
\to^\mathcal{S}
\frac{2\sqrt{2}}{\sqrt{\pi}}\int_0^T|\sigma_s|^3ds+3\sum_{q:T_q\leq T}\mathrm{sign}(\Delta X_{T_q})(\Delta X_{T_q})^2R_q
\end{align*}
as $n\to\infty$, where $\mathrm{sign}(x)=1$ if $x\geq0$; otherwise $\mathrm{sign}(x)=-1$.
\end{corollary}

If we consider $g(x)=x^3$, the following joint central limit theorem for the realized volatility and the cubic power variation is obtained:
\begin{corollary}\label{cor.skew}
Under Assumption \ref{asu2}, the variables
\begin{equation*}
\frac{1}{\sqrt{\Delta_n}}
\left(\mathrm{RV}^n_T(X)-[X,X]_T,\sum_{i=1}^{\lfloor T/\Delta_n\rfloor}(\Delta^n_iX)^3-\sum_{0\leq s\leq T}(\Delta X_s)^3\right)
\end{equation*}
converge stably in law to
\begin{equation*}
\left(\sqrt{2IQ_T}U^0+\mathcal{Z}_T(X,2),\mathcal{Z}_T(X,3)\right)
\end{equation*}
as $n\to\infty$.
\end{corollary}

We can use Corollary \ref{cor.skew} to derive the asymptotic distribution of the realized skewness: Combining Corollary \ref{cor.skew} with the delta method for stable convergence (Proposition 2(ii) of \cite{Podolskij2010}), we obtain the following result:
\begin{thm}\label{thm:skew}
Under Assumption \ref{asu2}, the variables
\[
\frac{1}{\sqrt{\Delta_n}}\left(\frac{1}{\lfloor T/\Delta_n\rfloor}RDSkew-\frac{\sum_{0\leq s\leq T}(\Delta X_s)^3}{\left([X,X]_T\right)^{3/2}}\right)
\]
converge stably in law to
\[
G_T:=\frac{[X,X]_T^{3/2}\mathcal{Z}_T(X,3)-\frac{3}{2}\sqrt{[X,X]_T}\sum_{0\leq s\leq T}(\Delta X_s)^3\left\{\sqrt{2IQ_T}U^0+\mathcal{Z}_T(X,2)\right\}}{[X,X]_T^3}
\]
as $n\to\infty$.
\end{thm}

\section{Microstructure noise and stochastic sampling}\label{section:noise}

It is widely recognized that modeling raw high-frequency financial data as \textit{direct} observations of an It\^o semimartingale $X$ is unrealistic. One common approach to deal with this issue is to assume that we observe the process $X$ with some measurement errors (referred to as \textit{microstructure noise}) rather than $X$ itself; see Chapter 7 of \cite{AJ2014} and references therein. Also, raw high-frequency financial data are typically recorded at stochastic sampling times, so the assumption that we observe data at equidistant sampling times is not applicable. Motivated by these reasons, in this section we consider an observed model which takes account of microstructure noise and stochastic sampling times, and develop an asymptotic theory for estimating \eqref{skewness} under such a situation.

Let us introduce the precise mathematical description of our model. We denote by $t^n_0,t^n_1,\dots$ the observation times which are assumed to be $(\mathcal{F}_t)$-stopping times and satisfy $t^n_i\uparrow\infty$ as $i\to\infty$. We also assume that
\begin{equation*}
r_n(t):=\sup_{i\geq0}(t^n_i\wedge t-t^n_{i-1}\wedge t)\to^P0
\end{equation*}
as $n\to\infty$ for any $t\in\mathbb{R}_+$, with setting $t^n_{-1}=0$ for notational convenience.

The observed process $Y$ is contaminated by some noise:
\begin{equation*}
Y_t=X_t+\epsilon_t.
\end{equation*}
The noise process $\epsilon$ implicitly depends on $n\in\mathbb{N}$ and is defined on a very good filtered extension $\mathcal{B}_n=(\Omega_n,\mathcal{F}^n,(\mathcal{F}^n_t)_{t\geq0} ,P_n)$ of $\mathcal{B}$ (see page 36 of \cite{JP2012} for the definition of very good filtered extensions). $\epsilon$ is an $(\mathcal{F}^n_t)$-optional process and,  conditionally on $\mathcal{F}$, the sequence $(\epsilon_{t^n_i})_{i=0}^\infty$ is independent and the (conditional) distribution of $\epsilon_{t^n_i}$ is given by $Q_{t^n_i(\omega)}(\omega,du)$ for each $i=0,1,\dots$, where for each $t\geq0$ $Q_t(\omega,du)$ denotes a transition probability from $(\Omega,\mathcal{F}_t)$ to $\mathbb{R}$. We assume that the $Q_t(\omega,du)$'s satisfy the following condition:
\begin{equation}\label{condition:Q}
\left.\begin{array}{l}
\int uQ_t(\omega,du)=0\text{ for every }t\geq0,\\
\text{the process $(Q_t(\cdot,A))_{t\geq0}$ is $(\mathcal{F}_t)$-progressively measurable for any Borel set $A$ of $\mathbb{R}$.}
\end{array}
\right\}
\end{equation}
A concrete construction of such a noise process can be found in Section 2 of \cite{Koike2015jclt}.

\subsection{Construction of estimators}

As was pointed out by \cite{LWL2013}, the realized skewness $RDSkew$ is no longer a consistent estimator for \eqref{skewness} in the presence of microstructure noise even after appropriate scaling. Hence we modify the realized skewness by the \textit{pre-averaging procedure}, which is a general scheme to remove the effects of microstructure noise from observation data; see \cite{PV2009a} and Chapter 16 of \cite{JP2012} for example.

First, we choose a sequence $k_n$ of positive integers and a number $\theta\in(0,\infty)$ such that
$k_n=\theta\Delta_n^{-1/2} +o(\Delta_n^{-1/4})$
as $n\to\infty$. We also choose a continuous function $g:[0,1]\rightarrow\mathbb{R}$ which is piecewise $C^1$ with a piecewise Lipschitz derivative $g'$ and satisfies $g(0)=g(1)=0$ and $\int_0^1 g(x)^3\mathrm{d}x\neq0.$
After that, for any process $V$ we define the variables
\begin{equation*}
\overline{V}_{i}=\sum_{p=1}^{k_n-1}g\left(\frac{p}{k_n}\right)\left(V_{t^n_{i+p}}-V_{t^n_{i+p-1}}\right),\qquad
i=0,1,\dots.
\end{equation*}
In the following we develop a central limit theorem for the Pre-averaged Realized Volatility
\[
\PRV^n_T=\frac{1}{\psi_2k_n}\sum_{i=0}^{N^n_T-k_n+1}\left(\overline{Y}_i\right)^2-\frac{\psi_1}{2\psi_2k_n}\sum_{i=1}^{N^n_T}(Y_{t^n_i}-Y_{t^n_{i-1}})^2,
\]
and the Pre-averaged realized Cubic power Variation
\[
\PCV^n_T=\frac{1}{\psi_3k_n}\sum_{i=0}^{N^n_T-k_n+1}\left(\overline{Y}_i\right)^3,
\]
where $N^n_T=\max\{i:t^n_i\leq T\}$ and
\[
\psi_1=\int_0^1g'(x)^2dx,\qquad
\psi_2=\int_0^1g(x)^2dx,\qquad
\psi_3=\int_0^1g(x)^3dx.
\]
For the equidistant sampling case $t^n_i=i\Delta_n$, it is known that
\[
\PRV^n_T\to^P[X,X]_T,\qquad
\PCV^n_T\to^P\sum_{0\leq s\leq T}(\Delta X_s)^3
\]
as $n\to\infty$ from Theorems 16.2.1 and 16.6.1 of \cite{JP2012}. Therefore, we may expect that
\[
\frac{\PCV^n_T}{\left(\PRV^n_T\right)^{3/2}}
\]
is a consistent estimator for \eqref{skewness}. Our aim is to derive the asymptotic distribution of the above statistic.


\subsection{Asymptotic results}

\subsubsection{Notation}
We write $X^n\xrightarrow{ucp}X$ for processes $X^n$ and $X$ to express shortly that $\sup_{0\leq t\leq T}|X^n_t-X_t|\rightarrow^p0$. $\varpi$ denotes some (fixed) positive constant. We denote by $(\mathcal{G}_t)$ the smallest filtration containing $(\mathcal{F}_t)$ such that $\mathcal{G}_0$ contains the $\sigma$-field generated by $\mu$, i.e.~the $\sigma$-field generated by all the variables $\mu(A)$, where $A$ ranges all measurable subsets of $\mathbb{R}_+\times E$.

For any real-valued bounded measurable functions $u,v$ on $[0,1]$, we define the function $\phi_{u,v}$ on $[0,1]$ by
\[
\phi_{u,v}(y)=\int_y^1u(x-y)v(x)dx.
\]
Then we put
\begin{gather*}
\Phi_{22}=\int_{0}^{1}\phi_{g,g}(y)^2\mathrm{d}y,\qquad
\Phi_{12}=\int_{0}^{1}\phi_{g,g}(y)\phi_{g',g'}(y)\mathrm{d}y,\qquad
\Phi_{11}=\int_{0}^{1}\phi_{g',g'}(y)^2\mathrm{d}y
\end{gather*}
and
\begin{align*}
&\Phi_{3+}=\int_{0}^{1}\phi_{g,g^2}(y)^2\mathrm{d}y,\qquad
\Phi_{3-}=\int_{0}^{1}\phi_{g^2,g}(y)\mathrm{d}y,\\
&\Phi'_{3+}=\int_{0}^{1}\phi_{g',g^2}(y)^2\mathrm{d}y,\qquad
\Phi'_{3-}=\int_{0}^{1}\phi_{g^2,g'}(y)^2\mathrm{d}y,\\
&\Phi_{23+}=\int_{0}^{1}\phi_{g,g}(y)\phi_{g,g^2}(y)\mathrm{d}y,\qquad
\Phi_{23-}=\int_{0}^{1}\phi_{g,g}(y)\phi_{g^2,g}(y)\mathrm{d}y,\\
&\Phi'_{23+}=\int_{0}^{1}\phi_{g',g}(y)\phi_{g',g^2}(y)\mathrm{d}y,\qquad
\Phi'_{23-}=\int_{0}^{1}\phi_{g,g'}(y)\phi_{g^2,g'}(y)\mathrm{d}y.
\end{align*}

We define the process $\alpha$ by $\alpha(\omega)_t=\int u^2Q_t(\omega,du)$.

\subsubsection{Assumptions}

We enumerate the assumptions which are imposed to derive our limit theorem.
\begin{enumerate}[noitemsep,label={\normalfont[A1]}]

\item \label{hypo:A1} 
It holds that
\begin{equation}\label{A4}
r_n(t)=o_p(\Delta_n^{\xi})
\end{equation}
as $n\to\infty$ (note that $t^n_{-1}=0$ by convention) for every $t>0$ and every $\xi\in(0,1)$. Moreover, for each $n$ we have {a $(\mathcal{G}_t)$}-progressively measurable positive-valued process $G^n_t$ and a random subset $\mathcal{N}^n$ of $\mathbb{Z}_+$ satisfying the following conditions:

\begin{enumerate}[nosep,label=(\roman*),ref=\roman*]

\item $\{(\omega,p)\in\Omega\times\mathbb{Z}_+:p\in\mathcal{N}^n(\omega)\}$ is a measurable set of $\Omega\times\mathbb{Z}_+$. Moreover, there is a constant $\kappa\in(0,\frac{1}{2})$ such that $\#(\mathcal{N}^n\cap\{p:t^n_p\leq t\})=O_p(n^\kappa)$ as $n\to\infty$ for every $t>0$.

\item $E[\Delta_n^{-1}(t^n_{p+1}-t^n_p)\big|{\mathcal{G}_{t^n_p}}]=G^n_{t^n_p}$ for every $n$ and every $p\in\mathbb{Z}_+\setminus\mathcal{N}^n$.

\item \label{hypo:A1iii} There is a c\`adl\`ag $(\mathcal{F}_t)$-adapted positive valued process $G$ such that
\begin{enumerate}[nosep,label=(iii-\alph*)]

\item $\Delta_n^{-\varpi}(G^n-G)\xrightarrow{ucp}0$,

\item $G_{t-}>0$ for every $t>0$,

\item $G$ is an It\^o semimartingale of the form
\begin{align*}
\textstyle G_t=G_0+\int_0^t\widehat{b}_s\mathrm{d}s+\int_0^t\widehat{\sigma}_s\mathrm{d}W_s+\int_0^t\widehat{\sigma}'_s\mathrm{d}\widehat{W}_s
+\int_0^t\int_{|\widehat{\delta}(s,z)|\leq1\}}\widehat{\delta}(s,z)(\mu-\nu)(ds,dz)\\
\textstyle+\int_0^t\int_{|\widehat{\delta}(s,z)|>1}\widehat{\delta}(s,z)\mu(ds,dz),
\end{align*}
where $\widehat{b}_s$ is a locally bounded and $(\mathcal{F}_t)$-progressively measurable real-valued process,  $\widehat{\sigma}_s$ and $\widehat{\sigma}'_s$ are c\`adl\`ag $(\mathcal{F}_t)$-adapted processes, $\widehat{W}_s$ is an $(\mathcal{F}_t)$-standard Wiener process independent of $W$, and $\widehat{\delta}$ is an $(\mathcal{F}_t)$-predictable real-valued function on $\Omega\times\mathbb{R}_+\times E$ such that there is a sequence $(\widehat{\rho}_j)$ of $(\mathcal{F}_t)$-stopping times increasing to infinity and, for each $j$, a deterministic non-negative function $\widehat{\gamma}_j$ on $E$ satisfying $\int\widehat{\gamma}_j(z)^2\wedge1\lambda(\mathrm{d}z)<\infty$ and $|\widehat{\delta}(\omega,t,z)|\leq\widehat{\gamma}_j(z)$ for all $(\omega,t,z)$ with $t\leq\widehat{\rho}_j(\omega)$.

\end{enumerate}

\end{enumerate}
Furthermore, $\max_{p=0,1,\dots,N^n_T}E\left[\Delta_n^{-1}(t^n_{p+1}-t^n_{p})|\mathcal{F}_{t^n_{p}}\right]$ is tight as $n\to\infty$ for every $t>0$.

\end{enumerate}

\begin{enumerate}[label={\normalfont[A2]}]

\item \label{hypo:A2} The volatility process $\sigma$ is an It\^o semimartingale of the form
\begin{align*}
\textstyle\sigma_t=\sigma_0+\int_0^t\widetilde{b}_s\mathrm{d}s+\int_0^t\widetilde{\sigma}_s\mathrm{d}W_s+\int_0^t\widetilde{\sigma}'_s\mathrm{d}\widetilde{W}_s
+\int_0^t\int_{|\widetilde{\delta}(s,z)|\leq1\}}\widetilde{\delta}(s,z)(\mu-\nu)(ds,dz)\\
\textstyle+\int_0^t\int_{|\widetilde{\delta}(s,z)|>1}\widetilde{\delta}(s,z)\mu(ds,dz)
\end{align*}
where $\widetilde{b}_s$ is a locally bounded and $(\mathcal{F}_t)$-progressively measurable  process, $\widetilde{\sigma}_s$ and $\widetilde{\sigma}'_s$ are c\`adl\`ag {$(\mathcal{F}_t)$}-adapted processes, $\widetilde{W}_s$ is an $(\mathcal{F}_t)$-standard Wiener process independent of $W$, and $\widetilde{\delta}$ is an {$(\mathcal{F}_t)$}-predictable function on $\Omega\times\mathbb{R}_+\times E$.

Moreover, for each $j$ there is an {$(\mathcal{F}_t)$}-stopping time $\rho_j$, a bounded {$(\mathcal{F}_t)$}-progressively measurable process $b_s$, a deterministic non-negative function $\gamma_j$ on $E$, and a constant $\Lambda_j$ such that $\rho_j\uparrow\infty$ as $j\to\infty$ and, for each $j$,
\begin{enumerate}[nosep,label=(\roman*)]

\item  $b(\omega)_s=b(j)(\omega)_s$ if $s<\rho_j(\omega)$,

\item $E\left[|b(j)_{t_1}-b(j)_{t_2}|^2|\mathcal{F}_{t_1\wedge t_2}\right]
\leq \Lambda_j E\left[|t_1-t_2|^\varpi|\mathcal{F}_{t_1\wedge t_2}\right]$ for any $(\mathcal{F}_t)$-stopping times $t_1$ and $t_2$ bounded by $j$,

\item {$\int\left\{\gamma_j(z)^2\wedge1\right\}\lambda(\mathrm{d}z)<\infty$} and $|\delta(\omega,t,z)|\vee|\widetilde{\delta}(\omega,t,z)|\leq\gamma_j(z)$ for all $(\omega,t,z)$ with $t\leq\rho_j(\omega)$,

\item $E\left[|\delta(t_1\wedge\rho_j,z)-\delta(t_2\wedge\rho_j,z)|^2|\mathcal{F}_{t_1\wedge t_2}\right]
\leq \Lambda_j\gamma_j(z)^2E\left[|t_1-t_2|^\varpi|\mathcal{F}_{t_1\wedge t_2}\right]$ for any $(\mathcal{F}_t)$-stopping times $t_1$ and $t_2$ bounded by $j$.

\end{enumerate}

\end{enumerate}

\begin{enumerate}[label={\normalfont[A3]}]

\item \label{hypo:A3} $Q_t$'s satisfy \eqref{condition:Q} and there is a sequence $(\rho'_{j})_{j\geq1}$ of {$(\mathcal{F}_t)$}-stopping times increasing to infinity such that
\[
\sup_{\omega\in\Omega,t<\rho'_{j}(\omega)}\int u^6 Q_t(\omega,du)<\infty.
\]
Moreover, for each $j$ there is a bounded c\`adl\`ag {$(\mathcal{F}_t)$}-adapted process $\alpha(j)_t$ and a constant $\Lambda_j'$ such that
\begin{enumerate}[nosep,label=(\roman*)]

\item  $\alpha(j)(\omega)_t=\alpha(\omega)_t$ if $t<\rho'_j(\omega)$,

\item $E\left[|\alpha(j)_{t_1}-\alpha(j)_{t_2}|^2|\mathcal{F}_{t_1\wedge t_2}\right]
\leq \Lambda'_j E\left[|t_1-t_2|^\varpi|\mathcal{F}_{t_1\wedge t_2}\right]$ for any $(\mathcal{F}_t)$-stopping times $t_1$ and $t_2$ bounded by $j$.

\end{enumerate}

\end{enumerate}

\begin{enumerate}[label={\normalfont[A4]}]

\item \label{hypo:A4} A regular conditional probability of $P$ given $\mathcal{H}$ exists for any sub-$\sigma$-field $\mathcal{H}$ of $\mathcal{F}$.

\end{enumerate}

\begin{rmk}
Assumptions \ref{hypo:A1}--\ref{hypo:A4} are almost identical to the ones imposed to prove a central limit theorem for $\PRV^n_T$ in \cite{Koike2015jclt}. A few differences appear in \ref{hypo:A1} and \ref{hypo:A3}: We add an additional (mild) assumption on the tightness of random variables
\[
\max_{p=0,1,\dots,N^n_T}E\left[\Delta_n^{-1}(t^n_{p+1}-t^n_{p})|\mathcal{F}_{t^n_{p}}\right].
\]
to \ref{hypo:A1} (note that this assumption automatically follows from \ref{hypo:A1}(ii)--(iii) if $\mathcal{N}^n=\emptyset$). In the meantime, \ref{hypo:A3} requires the finiteness of the sixth moment of the noise process.
\end{rmk}

\begin{rmk}
Since Assumption \ref{hypo:A1} contains some non-standard aspects compared with common ones used in the literature, we briefly make some comments on it (see Remark 3.2 of \cite{Koike2015jclt} for more details). 
\begin{enumerate}

\item A typical example satisfying this assumption is the \textit{restricted discretization scheme} introduced in Chapter 14 of \cite{JP2012}, where the $t^n_p$'s are modeled as
\begin{equation*}
t^n_p=t^n_{p-1}+\theta^n_{t^n_{p-1}}\varepsilon^n_p,\qquad p=1,2,\dots,
\end{equation*}
with $\theta^n$ being a c\`adl\`ag $(\mathcal{F}_t)$-adapted process, $(\varepsilon^n_p)_{p\geq1}$ being a sequence of i.i.d.~positive variables independent of $b$, $\sigma$, $\delta$, $W$, $\mu$, and such that $E[\varepsilon^n_p]=1$ and $E[(\varepsilon^n_p)^r]<\infty$ for every $r>0$, and $t^n_0=0$. 
An appropriate construction of the filtration $(\mathcal{F}_t)$ allows us to assume the independence between $\epsilon^n_p$ and $\mathcal{F}_{t^n_{p-1}}$ for all $n,p$. 
In this case we have \ref{hypo:A1}(i)--(ii) by setting $\mathcal{N}^n=\emptyset$ and $G^n=\Delta_n^{-1}\theta^n$. Then, \ref{hypo:A1}(iii) corresponds to (a weaker version of) Assumption (E) of \cite{JP2012}, and \eqref{A4} follows from Lemma 14.1.5 of \cite{JP2012} (note that the last condition on tightness is automatically satisfied once \ref{hypo:A1}(iii) holds true since $\mathcal{N}^n=\emptyset$). 

\item The main reason why we introduce an involved assumption compared with the standard ones is that our assumption does not rule out the dependence between $\varepsilon^n_p$'s and $X$ (see Example 4.1 of \cite{Koike2015jclt} for instance). 
However, we remark that \ref{hypo:A1} rules out some kind of dependence between the observation times and the jumps of the observed process because we take the conditional expectation given $\mathcal{G}_{t^n_p}$'s instead of $\mathcal{F}_{t^n_p}$'s in \ref{hypo:A1}(ii). In fact, our assumption does not allow the case that $t^n_p$'s are given by hitting times of a pure-jump L\'evy process whose jump measure is $\mu$ (note that in this case the L\'evy process must have infinite activity jumps due to \eqref{A4}). We however note that our assumption does not exclude such a dependence completely; see Example 4.2 of \cite{Koike2015jclt} for instance. 

\item The set $\mathcal{N}^n$ can be interpreted as a set of exceptional indices $p$ for which the equation $E[\Delta_n^{-1}(t^n_{p+1}-t^n_p)\big|{\mathcal{G}_{t^n_p}}]=G^n_{t^n_p}$ does not hold true. This additional complexity is useful to ensure the stability of Assumption \ref{hypo:A1} under the localization used in the proof (see Lemma 6.1 of \cite{Koike2015jclt}). 
Non-empty $\mathcal{N}^n$ also excludes some trivial exceptions of \ref{hypo:A1} with $\mathcal{N}^n=\emptyset$, such as $t^n_0=\log n/n$ and $t^n_p=t^n_{p-1}+\Delta_n$ for $p\geq1$. 

\end{enumerate}
\end{rmk}

\subsubsection{Results}

\begin{thm}\label{mainthm}
Suppose that \ref{hypo:A1}--\ref{hypo:A4} are satisfied. Then
\[
\left(\Delta_n^{-1/4}\left(\PRV^n_T-[X,X]_T\right),\Delta_n^{-1/4}\left(\PCV^n_T-\sum_{0\leq s\leq T}(\Delta X_s)^3\right)\right)^*\to^\mathcal{S}\Gamma_T^{1/2}\zeta
\]
as $n\to\infty$, where $\zeta$ is a bivariate standard normal variable which is defined on an extension of $\mathcal{B}$ and independent of $\mathcal{F}$, and $\Gamma_T$ is the $\mathbb{R}^2\otimes\mathbb{R}^2$-valued variable given by
\[
\Gamma_T=
\left[
\begin{array}{cc}
\Gamma^c_T+\overline{\Gamma}_T^{11}  & \overline{\Gamma}_T^{12}  \\
\overline{\Gamma}_T^{12}  & \overline{\Gamma}_T^{22}
\end{array}
\right]
\]
with
\begin{align*}
\Gamma^c_T&=\frac{4}{\psi_2^{2}}\int_0^T\left[\Phi_{22}\theta\sigma^4_sG_s
+2\frac{\Phi_{12}}{\theta}\sigma^2_s\alpha_s
+\frac{\Phi_{11}}{\theta^3}\frac{\alpha_s^{2}}{G_s}\right]\mathrm{d}s,\\
\overline{\Gamma}^{11}_T&=\frac{4}{\psi_2^{2}}\sum_{0\leq s\leq T}(\Delta X_s)^2\left\{\Phi_{22}\theta\left(\sigma_s^2G_s+\sigma_{s-}^2G_{s-}\right)+\frac{\Phi_{12}}{\theta}\left(\alpha_s+\alpha_{s-}\right)\right\},\\
\overline{\Gamma}^{12}_T&=\frac{6}{\psi_2\psi_3}\sum_{0\leq s\leq T}(\Delta X_s)^3\left\{\theta\left(\Phi_{23+}\sigma_s^2G_s+\Phi_{23-}\sigma_{s-}^2G_{s-}\right)+\theta^{-1}\left(\Phi'_{23+}\alpha_s+\Phi'_{23-}\alpha_{s-}\right)\right\},\\
\overline{\Gamma}^{22}_T&=\frac{9}{\psi_3^{2}}\sum_{0\leq s\leq T}(\Delta X_s)^4\left\{\theta\left(\Phi_{3+}\sigma_s^2G_s+\Phi_{3-}\sigma_{s-}^2G_{s-}\right)+\theta^{-1}\left(\Phi'_{3+}\alpha_s+\Phi'_{3-}\alpha_{s-}\right)\right\}.
\end{align*}
\end{thm}

A proof of the above result is given in Section \ref{proof:mainthm}. Combining the above result with the delta method for stable convergence, we obtain the asymptotic distribution of $\PCV^n_T/\left(\PRV^n_T\right)^{3/2}$ as follows:
\begin{thm}\label{thm:noisy-skew}
Under the assumptions of Theorem \ref{mainthm}, we have
\[
\Delta_n^{-1/4}\left(\frac{\PCV^n_T}{\left(\PRV^n_T\right)^{3/2}}-\frac{\sum_{0\leq s\leq T}(\Delta X_s)^3}{([X,X]_T)^{3/2}}\right)\to^\mathcal{S}\sqrt{d_{1,T}^2(\Gamma^c_T+\overline{\Gamma}_T^{11})+2d_{1,T}d_{2,T}\overline{\Gamma}_T^{12}+d_{2,T}^2\overline{\Gamma}_T^{22}}\times\zeta
\]
as $n\to\infty$, where $\zeta$ is a standard normal variable which is defined on an extension of $\mathcal{B}$ and independent of $\mathcal{F}$, and
\[
d_{1,T}=-\frac{3}{2}\frac{\sum_{0\leq s\leq T}(\Delta X_s)^3}{([X,X]_T)^{5/2}},\qquad
d_{2,T}=\frac{1}{([X,X]_T)^{3/2}}.
\]
\end{thm}

\begin{rmk}
The stable convergence result on the first component $\Delta_n^{-1/4}\left(\PRV^n_T-[X,X]_T\right)$ in Theorem \ref{mainthm} is a special case of Theorem 3.1 from \cite{Koike2015jclt}. 
In contrast, the stable convergence result on the second component 
\[
\Delta_n^{-1/4}\left(\PCV^n_T-\sum_{0\leq s\leq T}(\Delta X_s)^3\right)
\]
in Theorem \ref{mainthm} is new even in the equidistant sampling case $t^n_i=i\Delta_n$. As is remarked in the Introduction, Theorem 16.3.1 of \cite{JP2012} deals with the asymptotic distribution of the statistic
\[
\frac{1}{k_n}\sum_{i=0}^{N^n_T-k_n+1}f(\overline{Y}_i)
\]
in the equidistant sampling case for a $C^2$ function $f:\mathbb{R}\to\mathbb{R}$ which is a linear combinations of positively homogeneous functions with degree (strictly) bigger than 3. Since the cubic function $f(x)=x^3$ does not have this property, Theorem 16.3.1 of \cite{JP2012} is not applicable to deriving the asymptotic distribution of $\PCV^n_T$.  
\end{rmk}

\begin{rmk}
The main difference between the asymptotic distributions given in Theorems \ref{thm2}--\ref{thm:skew} and Theorems \ref{mainthm}--\ref{thm:noisy-skew} is that the former ones are in general \texttt{not} $\mathcal{F}$-conditionally Gaussian due to the additional randomness caused by the uniform variables $U_q,U_q'$, while the latter ones are always $\mathcal{F}$-conditionally Gaussian. 
This is a byproduct of the pre-averaging procedure and commonly observed in the literature of pre-averaging estimators for functionals of jumps. 
\end{rmk}

\begin{rmk}
If $X$ is continuous, Theorem \ref{mainthm} implies $\Delta_n^{-1/4}\PCV^n_T\to^P0$ as $n\to\infty$, hence we will need a larger scaling factor than $\Delta_n^{-1/4}$ to obtain a non-degenerate asymptotic distribution of $\PCV^n_T$. 
To the best of our knowledge, nothing is known about the non-trivial asymptotic distribution of $\PCV^n_T$ even in the equidistant sampling setting. 
An analogy to the non-noisy case suggests that the proper scaling factor is $\Delta_n^{-1/2}$ and $\Delta_n^{-1/2}\PCV^n_T$ would converge stably in law to a mixed normal distribution with a non-zero conditional mean (cf.~Example 6 of \cite{Kinnebrock2008} for the equidistant sampling setting and Theorem 2 of \cite{LMRZZ2014} for an irregular sampling setting). 
\end{rmk}

\begin{rmk}
We note that the main reason why the asymptotic distribution of the realized skewness obtained in the previous section is centered is because we have $\sqrt{\Delta_n}\sum_{i=1}^{N^n_T}(X^c_{t^n_i}-X^c_{t^n_{i-1}})^3\to^P0$ as $n\to\infty$ \textit{as long as $t^n_i=i\Delta_n$}. This is in general not true when we consider more general sampling schemes as $(t^n_i)$ such that they depend on the process $X^c$. In particular, Assumption \ref{hypo:A1} does not rule out situations where the variables $\sqrt{\Delta_n}\sum_{i=1}^{N^n_T}(X^c_{t^n_i}-X^c_{t^n_{i-1}})^3$ converge in probability to some non-zero random variable as $n\to\infty$; see e.g.~Example 3.2 from \cite{Koike2015} and Example 5 from \cite{LMRZZ2014}. In such a situation we conjecture that the asymptotic distribution of the realized skewness estimator would be no longer centered. In contrast, Theorem \ref{thm:noisy-skew} tells us that the estimator $\PCV^n_T/\left(\PRV^n_T\right)^{3/2}$ is asymptotically centered even under \ref{hypo:A1}. This is another byproduct of the pre-averaging procedure.
\end{rmk}

\begin{rmk}
In the absence of noise the simultaneous presence of jumps and the randomness of observation times typically adds more complexity to the asymptotic distribution of statistics of the form \eqref{functional} more complex, as seen in \cite{BV2014}, \cite{VZ2016} and \cite{MV2016}. 
In this sense the result of Theorem \ref{mainthm} again contrasts with non-noisy cases because the estimators $\PRV^n_T$ and $\PCV^n_T$ are asymptotically mixed normal even with stochastic sampling times. This is also a byproduct of the pre-averaging procedure.
\end{rmk}

\section{Proofs}\label{section:proofs}

\subsection{Proof of Theorem \ref{thm2}}\label{proof:thm2}

First of all, a standard localization procedure, described in detail in Lemma 4.49 of \cite{JP2012}, for instance, allows us to assume that there are a positive constant $A$ and a nonnegative deterministic measurable function $\gamma$ on $E$ such that
\begin{equation}\label{SH}
\left.\begin{array}{l}
|b(\omega)_t|\leq A,\qquad|\sigma(\omega)_t|\leq A,\qquad|X(\omega)_t|\leq A,\qquad|X^c(\omega)_t|\leq A,\\
|\delta(\omega,t,z)|\leq\gamma(z)\leq A,\qquad
\int\gamma(z)^2\lambda(dz)\leq A.
\end{array}\right\}
\end{equation}

The strategy of the proof is the same as the one used in the proof of Theorem 5.1.2 from \cite{JP2012}, and we divide the proof into several steps. For the part corresponding to Steps 1--3 of \cite{JP2012}'s proof, we can adopt almost the same argument as the original one, hence it is just briefly sketched in Step 2. In the remainder steps we will need an argument which is somewhat different from theirs.

Throughout the discussions, for random variables $X$ and $Y$ which may depend on the parameters $n,m,i$, $X\lesssim Y$ means that there exists a (non-random) constant $K>0$ independent of $n,m,i$ such that $X\leq KY$ a.s.\vspace{2.5mm}

\noindent\textit{Step 1)} We begin with introducing some notations. For each $m\in\mathbb{N}$, set $A_m=\{z:\gamma(z)>1/m\}$. Noting that $\nu(A_m)<\infty$, we denote by $(S(m,j))_{j\geq1}$ the successive jump times of the Poisson process $(\mu((0,t]\times(A_m\setminus A_{m-1})))_{t\geq0}$. Let $(S_p)_{p\geq1}$ be a reordering of the double sequence $(S(m,j))$, and we denote by $\mathcal{P}_m$ the set of all indices $p$ such that $S_p=S(m',j)$ for some $j\geq1$ and some $m'\leq m$. In the light of Proposition 5.1.1 from \cite{JP2012}, we may assume that $(S_p)=(T_q)$ without loss of generality.

Set
\begin{equation*}
\begin{array}{l}
b(m)_t=b_t-\int_{A_m\cap\{z:|\delta(t,z)|\leq1\}}\delta(t,z)\lambda(dz),\quad
C(m)_t=X_0+\int_0^tb(m)_sds+X_t^c,\\
X(m)_t=C(m)_t+\int_0^t\int_{A_m^c}\delta(s,z)(\mu-\nu)(ds,dz),\\
X'(m)_t=X_t-X(m)_t=\int_0^t\int_{A_m}\delta(s,z)\mu(ds,dz),\\
X''(m)_t=X(m)_t-X^c_t=X_0+\int_0^tb(m)_s ds+\int_0^t\int_{A_m^c}\delta(s,z)(\mu-\nu)(ds,dz),\\
Y(m)_t=X_t-\int_0^t\int_{A_m^c}\delta(s,z)(\mu-\nu)(ds,dz)=C(m)_t+X'(m)_t,
\end{array}
\end{equation*}
and denote by $\Omega_n(T,m)$ the set of all $\omega$ such that each interval $[0,T]\cap((i-1)\Delta_n,i\Delta_n]$ contains at most one jump of $X'(m)(\omega)$.
Note that the notations here are consistent with those from Eq.(5.1.10) of \cite{JP2012}.
Since $X'(m)$ has finitely many jumps on $[0,T]$, it holds that $P(\Omega_n(T,m))\to1$ as $n\to\infty$. We also set $i^n_p=\lceil S_p/\Delta_n\rceil$ so that $S_p$ is in $((i^n_p-1)\Delta_n,i^n_p\Delta_n]$. Here, for a real number $x$, $\lceil x\rceil$ denotes the minimum integer $l$ satisfying $l\geq x$. Then we define
\begin{align*}
R(n,p)=\frac{1}{\sqrt{\Delta_n}}\left(\Delta^n_{i^n_p}X-\Delta X_{S_p}\right),\qquad
\zeta^n_p=\frac{1}{\sqrt{\Delta_n}}\left(g(\Delta^n_{i^n_p}X)-g(\Delta X_{S_p})-g(\sqrt{\Delta_n}R(n,p))\right)
\end{align*}
and $Y^n_T(m)=\sum_{p\in\mathcal{P}_m:S_p\leq\Delta_n\lfloor T/\Delta_n\rfloor}\zeta^n_p.$
Moreover, for any semimartingale $S$, we set
\begin{align*}
\overline{\mathcal{V}}^{n}_T(S,g)=\frac{1}{\sqrt{\Delta_n}}\left(\mathcal{V}^n_T(S,g)-\sum_{0\leq s\leq T}g(\Delta S_s)\right),\qquad
\overline{Z}^n_T(S)=\frac{1}{\sqrt{\Delta_n}}\left(\text{RV}^n_T(S)-[S,S]_{\lfloor T/\Delta_n\rfloor\Delta_n}\right).
\end{align*}

\noindent\textit{Step 2)} 
First we fix $m$. With $p$ fixed, the sequence $R(n,p)$ is tight due to Proposition 4.4.10 of \cite{JP2012}. Therefore, we have $g(\sqrt{\Delta_n}R(n,p))/\sqrt{\Delta_n}\to^P0$ as $n\to\infty$ because $g(x)=O(x^2)$ as $x\to0$. On the other hand, repeated applications of the fundamental theorem of calculus yield
\begin{align*}
g(\Delta^n_{i_p}X)-g(\Delta X_{S_p})
=g'(\Delta X_{S_p})\sqrt{\Delta_n}R(n,p)
+\int_{0}^{\sqrt{\Delta_n}R(n,p)}\int_0^ug''(\Delta X_{S_p}+v)dvdu.
\end{align*}
Since $X$ is bounded and $g''$ is continuous, we have
\begin{align*}
\frac{1}{\sqrt{\Delta_n}}\int_{0}^{\sqrt{\Delta_n}R(n,p)}\int_0^ug''(\Delta X_{S_p}+v)dvdu\to^P0
\end{align*}
as $n\to\infty$. Consequently, we conclude that $\zeta^n_p-g'(\Delta X_{S_p})R(n,p)\to^P0$ as $n\to\infty$. On the other hand, an argument similar to the proof of Lemma 5.4.10 from \cite{JP2012} implies that 
$\left(\overline{Z}^n_T(C(m)),\left(R(n,p)\right)_{p\geq1}\right)
\to^{\mathcal{S}}\left(\sqrt{2IQ_T}U^0,\left(R_p\right)_{p\geq1}\right)$ 
as $n\to\infty$.  
Hence it holds that
\begin{align*}
\left(\overline{Z}^n_T(C(m)),\left(\Delta X_{S_p}R(n,p)\right)_{p\geq1},\left(\zeta^n_p\right)_{p\geq1}\right)
\to^{\mathcal{S}}\left(\sqrt{2IQ_T}U^0,\left(\Delta X_{S_p}R_p\right)_{p\geq1},\left(\zeta_p\right)_{p\geq1}\right)
\end{align*}
as $n\to\infty$, where $\zeta_p=g'(\Delta X_{S_p})R_p$. Since the set $\{S_p:p\in\mathcal{P}_m\}\cap[0,T]$ is finite, it follows that
\begin{align*}
\left(\overline{Z}^n_T(C(m)),\mathcal{Z}^n_T(m),Y^n(m)_T\right)
\to^{\mathcal{S}}\left(\sqrt{2IQ_T}U^0,\mathcal{Z}_T(X'(m),2),\overline{\mathcal{V}}_T(X'(m),g)\right)
\end{align*}
as $n\to\infty$, where $\mathcal{Z}^n_T(m)_T=2\sum_{p\in\mathcal{P}_m:S_p\leq\Delta_n\lfloor T/\Delta_n\rfloor}\Delta X_{S_p}R(n,p).$ 
Furthermore, the same argument as the last part of the proof of Lemma 5.4.10 from \cite{JP2012} implies that $\overline{Z}^n_T(C(m))+\mathcal{Z}^n_T(m)_T-\overline{Z}^n_T(Y(m))\to^P0$ as $n\to\infty$. Consequently, we conclude that
\[
\left(\overline{Z}^n_T(Y(m)),Y^n_T(m)\right)\to^\mathcal{S}(\sqrt{2IQ_T}U^0+\mathcal{Z}_T(X'(m),2),\overline{\mathcal{V}}_T(X'(m),g))
\]
as $n\to\infty$. 

Next we vary $m$. We can prove
\begin{equation*}
\overline{\mathcal{V}}_T(X'(m),g)\to^P\overline{\mathcal{V}}_T(X,g)
\end{equation*}
as $m\to\infty$ by the same argument as the proof of Eq.(5.1.16) from \cite{JP2012} and
\begin{equation*}
\left.\begin{array}{l}
\lim_{m\to\infty}\limsup_{n\to\infty}P\left(\left|\overline{Z}^n_T(Y(m))-\overline{Z}^n_T(X)\right|>\eta\right)\to0,\\
\mathcal{Z}_T(X'(m))\to^P\mathcal{Z}_T(X)\quad\mathrm{as}\quad m\to\infty
\end{array}\right\}
\end{equation*}
for any $\eta>0$ by the same argument as in the proof of Lemma 5.4.12 from \cite{JP2012} where our $Y(m)$ is denoted by $X(m)$.

Now, noting that $[X,X]_{\lfloor T/\Delta_n\rfloor\Delta_n}-[X,X]_{T}=o_P(\sqrt{\Delta_n})$ and the decomposition
\begin{equation*}
\overline{\mathcal{V}}^{n}_T(X,g)=\overline{\mathcal{V}}^{n}_T(X(m),g)+Y^n_T(m)
\end{equation*}
holding on the set $\Omega_n(T,m)$ as well as $\lim_{n\to\infty}P(\Omega_n(T,m))=1$ for every $m\in\mathbb{N}$, the proof of the theorem is completed once we show that
\begin{equation}\label{aim}
\lim_{m\to\infty}\limsup_{n\to\infty}P\left(\Omega_n(T,m)\cap\left\{\left|\overline{\mathcal{V}}^{n}_T(X(m),g)-\frac{1}{\sqrt{\Delta_n}}\mathcal{V}^{n}_T(X^c,g)\right|>\eta\right\}\right)=0
\end{equation}
for any $\eta>0$.

\noindent\textit{Step 3)} We begin by showing three inequalities used in the proof. The first and the second ones are elementary: if $\rho\in(0,2]$, the Lyapunov and Doob inequalities as well as $(\ref{SH})$ yield
\begin{align*}
E\left[\sup_{(i-1)\Delta_n\leq t\leq i\Delta_n}\left|\int_{(i-1)\Delta_n}^{t}\int_{A^c_m}\delta(s,z)(\mu-\nu)(ds,dz)\right|^\rho\right]
\leq\left(4\Delta_n\overline{\gamma}_m\right)^{\rho/2},
\end{align*}
where $\overline{\gamma}_m:=\int_{A_m^c}\gamma(z)^2\lambda(dz)$. Therefore, noting that
$\left|\int_{A_m\cap\{z||\delta(t,z)|\leq1\}}\delta(t,z)\nu(dz)\right|\leq Am,$
there exists a positive constant $K_\rho$ such that
\begin{equation}\label{jump.est}\textstyle
E\left[\sup_{(i-1)\Delta_n\leq t\leq i\Delta_n}\left|X''(m)_t-X''(m)_{(i-1)\Delta_n}\right|^\rho\right]\leq K_\rho\left\{\left(m\Delta_n\right)^\rho+\left(\Delta_n \overline{\gamma}_m\right)^{\rho/2}\right\}
\end{equation}
for every $i,n,m$.
On the other hand, for every $\rho\geq1$ there exists a constant $K'_\rho$ such that
\begin{equation}\label{cont.est}\textstyle
E\left[\sup_{(i-1)\Delta_n\leq t\leq i\Delta_n}\left|X^c_t-X^c_{(i-1)\Delta_n}\right|^\rho\right]\leq K'_\rho\Delta_n^{\rho/2}
\end{equation}
for every $i,n$, due to the Burkholder-Davis-Gundy inequality and $(\ref{SH})$.

Now we prove the third one. By using integration by parts repeatedly we obtain
\begin{align*}
\Delta^n_iX^{c}\left(\Delta^n_iX''(m)\right)^2
&{\textstyle=2\int_{(i-1)\Delta_n}^{i\Delta_n}(X^{c}_s-X^{c}_{(i-1)\Delta_n})(X''(m)_s-X''(m)_{(i-1)\Delta_n})dX''(m)_s}\\
&{\textstyle\quad+\int_{(i-1)\Delta_n}^{i\Delta_n}\int_{A^c_m}(X^{c}_s-X^{c}_{(i-1)\Delta_n})\delta(s,z)^2\mu(ds,dz)}\\
&{\textstyle\quad+\int_{(i-1)\Delta_n}^{i\Delta_n}(X''(m)_s-X''(m)_{(i-1)\Delta_n})^2dX^{c}_s}\\
&=:2\mathbb{I}^n_i+\mathbb{II}^n_i+\mathbb{III}^n_i.
\end{align*}
First consider $\mathbb{I}^n_i$. We decompose it as
\begin{align*}
\mathbb{I}^n_i
&{\textstyle=\int_{(i-1)\Delta_n}^{i\Delta_n}(X^{c}_s-X^{c}_{(i-1)\Delta_n})(X''(m)_s-X''(m)_{(i-1)\Delta_n})b(m)_s ds}\\
&{\textstyle\quad+\int_{(i-1)\Delta_n}^{i\Delta_n}\int_{A^c_m}(X^{c}_s-X^{c}_{(i-1)\Delta_n})(X''(m)_s-X''(m)_{(i-1)\Delta_n})\delta(s,z)(\mu-\nu)(ds,dz)}\\
&=:\mathbb{I}^{n,1}_i+\mathbb{I}^{n,2}_i.
\end{align*}
The Schwarz inequality and $(\ref{jump.est})$--$(\ref{cont.est})$ yield
\begin{align*}
E\left[\left|\mathbb{I}^{n,1}_i\right|\right]
&\leq Cm\int_{(i-1)\Delta_n}^{i\Delta_n}E\left[\left|(X^{c}_s-X^{c}_{(i-1)\Delta_n})(X''(m)_s-X''(m)_{(i-1)\Delta_n})\right|\right]ds\\
&\lesssim m\Delta_n^{3/2}\left(m\Delta_n+\sqrt{\overline{\gamma}_m\Delta_n}\right).
\end{align*}
On the other hand, since integration by parts implies that
\begin{align*}
&(X^{c}_s-X^{c}_{(i-1)\Delta_n})(X''(m)_s-X''(m)_{(i-1)\Delta_n})\\
=&\int_{(i-1)\Delta_n}^s(X^{c}_u-X^{c}_{(i-1)\Delta_n})dX''(m)_u+\int_{(i-1)\Delta_n}^s(X''(m)_u-X''(m)_{(i-1)\Delta_n})dX^{c}_u,
\end{align*}
the Doob inequality, $(\ref{SH})$ and $(\ref{jump.est})$--$(\ref{cont.est})$ imply that
\begin{align*}
E\left[\sup_{(i-1)\Delta_n\leq s\leq i\Delta_n}\left|(X^{c}_s-X^{c}_{(i-1)\Delta_n})(X''(m)_s-X''(m)_{(i-1)\Delta_n})\right|^2\right]
\lesssim \Delta_n^2\left(m^2\Delta_n+\overline{\gamma}_m\right),
\end{align*}
hence the Lyapunov and Doob inequalities yield
\begin{align*}
E\left[\left|\mathbb{I}^{n,2}_i\right|\right]
&\leq \left\{4\overline{\gamma}_m\int_{(i-1)\Delta_n}^{i\Delta_n}E\left[\left|(X^{c}_s-X^{c}_{(i-1)\Delta_n})(X''(m)_s-X''(m)_{(i-1)\Delta_n})\right|^2\right]ds\right\}^{1/2}\\
&\lesssim \Delta_n^{3/2}\left(m\sqrt{\overline{\gamma}_m\Delta_n}+\overline{\gamma}_m\right).
\end{align*}
Consequently, it holds that $E\left[|\mathbb{I}^n_i|\right]\lesssim \Delta_n^{3/2}\left(m^2\Delta_n+m\sqrt{\overline{\gamma}_m\Delta_n}+\overline{\gamma}_m\right)$. On the other hand, since $\nu$ is the compensator of $\mu$, we have
\begin{align*}
E\left[|\mathbb{II}^n_i|\right]
\leq E\left[\int_{(i-1)\Delta_n}^{i\Delta_n}\int_{A^c_m}\left|X^{c}_s-X^{c}_{(i-1)\Delta_n}\right|\delta(s,z)^2\nu(ds,dz)\right]
\lesssim \Delta_n^{3/2}\overline{\gamma}_m
\end{align*}
by $(\ref{SH})$ and $(\ref{cont.est})$, whereas the Davis inequality, $(\ref{SH})$ and $(\ref{jump.est})$ imply that
\begin{align*}
E\left[|\mathbb{III}^n_i|\right]
&{\textstyle\lesssim E\left[\left\{\int_{(i-1)\Delta_n}^{i\Delta_n}(X''(m)_s-X''(m)_{(i-1)\Delta_n})^4\sigma^2_sds\right\}^{1/2}\right]}\\
&\lesssim {\textstyle \Delta_n^{1/2}E\left[\sup_{(i-1)\Delta_n\leq s\leq i\Delta_n}(X''(m)_s-X''(m)_{(i-1)\Delta_n})^2\right]}
\lesssim \Delta_n^{3/2}\left(m^2\Delta_n+\overline{\gamma}_m\right).
\end{align*}
After all, there exists a positive constant $K''$ such that
\begin{equation}\label{IBP}
E\left[\left|\Delta^n_iX^{c}\left(\Delta^n_iX''(m)\right)^2\right|\right]\leq K''\Delta_n^{3/2}\left(m^2\Delta_n+m\sqrt{\overline{\gamma}_m\Delta_n}+\overline{\gamma}_m\right)
\end{equation}
for every $i,n,m$.\vspace{2.5mm}

\noindent\textit{Step 4)} Setting $k(x,y)=g(x+y)-g(x)-g(y)$, we have
\begin{align*}\textstyle
\overline{\mathcal{V}}^{n}_T(X(m),g)-\frac{1}{\sqrt{\Delta_n}}\mathcal{V}^{n}_T(X^c,g)
=\frac{1}{\sqrt{\Delta_n}}\sum_{i=1}^{\lfloor T/\Delta_n\rfloor}k(\Delta^n_iX^c,\Delta^n_iX''(m))
+\overline{\mathcal{V}}^{n}_T(X''(m),g)
\end{align*}
because $X(m)=X^c+X''(m)$. Therefore, the proof is completed once we verify the following equations for any $\eta>0$:
\begin{align}
&\lim_{m\to\infty}\limsup_{n\to\infty}P\left(\Omega_n(T,m)\cap\left\{\frac{1}{\sqrt{\Delta_n}}\left|\sum_{i=1}^{\lfloor T/\Delta_n\rfloor}k(\Delta^n_iX^c,\Delta^n_iX''(m))\right|>\eta\right\}\right)=0,\label{aim1.5}\\
&\lim_{m\to\infty}\limsup_{n\to\infty}P\left(\Omega_n(T,m)\cap\left\{\left|\overline{\mathcal{V}}^{n}_T(X''(m),g)\right|>\eta\right\}\right)=0.\label{aim2o}
\end{align}
This step is devoted to the proof of $(\ref{aim1.5})$. First, since $g(0)=g'(0)=0$ and $g''(x)=O(|x|)$ as $x\to0$, there exists a positive constant $\alpha$ such that
\begin{equation}\label{g.est}
|x|\leq 6A \Rightarrow |g(x)|\leq \alpha|x|^3,\quad |g'(x)|\leq \alpha|x|^2,\quad |g''(x)|\leq \alpha|x|.
\end{equation}
Next, using the fundamental theorem of calculus repeatedly, we have
\begin{align*}
k(x,y)=\int_0^y\int_0^xg''(v+u)dvdu,
\end{align*}
hence $(\ref{g.est})$ yields
\begin{equation}\label{k.est}
|x|+|y|\leq 6A \Rightarrow |k(x,y)|\leq \alpha\left(|x|^2|y|+|x||y|^2\right).
\end{equation}
Let us recall Since $|\Delta^n_iX^c|\leq 2A$ and $|\Delta^n_iX(m)|\leq|\Delta^n_iX|+|\Delta^n_iX'(m)|\leq 3A$ on $\Omega_n(T,m)$ due to $(\ref{SH})$, we obtain
\begin{align*}
&E\left[\left|\frac{1}{\sqrt{\Delta_n}}\sum_{i=1}^{\lfloor T/\Delta_n\rfloor}k(\Delta^n_iX^c,\Delta^n_iX''(m))\right|1_{\Omega_n(T,m)}\right]\\
\lesssim &\Delta_n^{-3/2}\left\{\Delta_n\left(m\Delta_n+\sqrt{\overline{\gamma}_m\Delta_n}\right)+\Delta_n^{3/2}\left(m^2\Delta_n+m\sqrt{\overline{\gamma}_m\Delta_n}+\overline{\gamma}_m\right)\right\}
\end{align*}
by $(\ref{k.est})$, the Schwarz inequality and $(\ref{jump.est})$--$(\ref{IBP})$. Therefore, noting that $\overline{\gamma}_m\to0$ as $m\to\infty$ because of $(\ref{SH})$ and the dominated convergence theorem and that  $\Delta X(m)_s=\Delta X''(m)_s$ as well as $\Delta X^c_s=0$ for all $s\geq0$, $(\ref{aim1.5})$ has been shown.

\noindent\textit{Step 5)} Now we prove $(\ref{aim2o})$ and complete the proof of the theorem. First, set $\phi(x,y)=k(x,y)-g'(x)y$. If $|y|>|x|$, $(\ref{g.est})$ and $(\ref{k.est})$ yield
\begin{equation}\label{phi.est}
|x|\leq 5A, |y|\leq A\Rightarrow|\phi(x,y)|\leq 3\alpha|x||y|^2,
\end{equation}
whereas repeated applications of the fundamental theorem of calculus imply that $\phi(x,y)=\int_0^y\int_0^ug''(x+v)dvdu-g(y)$, hence, if $|y|\leq |x|$, by $(\ref{g.est})$ we have
\begin{align*}
|x|\leq 5A, |y|\leq A\Rightarrow|\phi(x,y)|\leq\alpha(|x|+|y|)|y|^2+\alpha|y|^3\leq 3\alpha|x||y|^2,
\end{align*}
and thus $(\ref{phi.est})$ holds true.

Next, for any $i$, an application of It\^o's formula to the process $\Xi(m,i)_t=\int_0^t1_{\{s>(i-1)\Delta_n\}}dX''(m)_s$ and the function $g$ yields
\begin{align*}
g(\Xi(m,i)_t)&\textstyle=\int_{(i-1)\Delta_n}^tg'(\Xi(m,i)_{s-})dX''(m)_s\\
&\textstyle\quad+\sum_{(i-1)\Delta_n< s\leq t}\left\{\phi(\Xi(m,i)_{s-},\Delta X(m)_s)+g(\Delta X(m)_s)\right\}
\end{align*}
for any $t>(i-1)\Delta_n$. Therefore, noting that $\sum_{(i-1)\Delta_n< s\leq t}g(\Delta X(m)_s)$ is well-defined by assumption, for any $t>(i-1)\Delta_n$ we have
\begin{align}
&\hphantom{=}g(\Xi(m,i)_t)-\sum_{(i-1)\Delta_n< s\leq t}g(\Delta X(m)_s)\nonumber\\
&\int_{(i-1)\Delta_n}^tg'(\Xi(m,i)_{s-})dX''(m)_s
+\int_{(i-1)\Delta_n}^t\int_{A_m^c}\phi(\Xi(m,i)_{s-},\delta(s,z))\mu(ds,dz)\nonumber\\
&=\int_{(i-1)\Delta_n}^ta(n,m,i)_udu+\int_{(i-1)\Delta_n}^t\int_{A_m^c}k(\Xi(m,i)_{s-},\delta(s,z))(\mu-\nu)(ds,dz)\nonumber\\
&=:A(n,m,i)_t+M(n,m,i)_t,\label{JP5.1.19}
\end{align}
where
$a(n,m,i)_u=g'(\Xi(m,i)_{u})b(m)_u+\int_{A^c_m}\phi(\Xi(m,i)_{u},\delta(u,z))\lambda(dz).$
Note that $A(n,m,i)$ and $M(n,m,i)$ are well-defined due to $(\ref{SH})$ and $(\ref{k.est})$--$(\ref{phi.est})$.

Let us set $T(n,m,i)=\inf\{s>(i-1)\Delta_n:|\Xi(m,i)_s|>5A\}$. On the set $\Omega_n(T,m)$ we have $|\Xi(m,i)_s|\leq5A$ for all $s\leq T$ and $i\leq\lfloor T/\Delta_n\rfloor$ due to the decomposition $X''(m)=X-X^c-X'(m)$ and $(\ref{SH})$, hence $T(n,m,i)>i\Delta_n$. 
Thus, in view of $(\ref{JP5.1.19})$, we have
\begin{align}
&P\left(\Omega_n(T,m)\cap\left\{\left|\overline{\mathcal{V}}^{n}_T(X''(m),g)\right|>\eta\right\}\right)\nonumber\\
&\leq P\left(\frac{1}{\sqrt{\Delta_n}}\sum_{i=1}^{\lfloor T/\Delta_n\rfloor}|A(n,m,i)_{(i\Delta_n)\wedge T(n,m,i)}|>\frac{\eta}{2}\right)
+P\left(\frac{1}{\sqrt{\Delta_n}}\left|\sum_{i=1}^{\lfloor T/\Delta_n\rfloor}M(n,m,i)_{(i\Delta_n)\wedge T(n,m,i)}\right|>\frac{\eta}{2}\right).\label{tnmi}
\end{align}
Therefore, in order to prove $(\ref{aim2o})$ it suffices to show that
\begin{equation}\label{aim3o}
\left.\begin{array}{l}
\lim_{m\to\infty}\limsup_{n\to\infty}\frac{1}{\sqrt{\Delta_n}}E\left[\sum_{i=1}^{\lfloor T/\Delta_n\rfloor}|A(n,m,i)_{(i\Delta_n)\wedge T(n,m,i)}|\right]=0,\\
\lim_{m\to\infty}\limsup_{n\to\infty}\frac{1}{\Delta_n}E\left[\sum_{i=1}^{\lfloor T/\Delta_n\rfloor}\langle M(n,m,i)\rangle_{(i\Delta_n)\wedge T(n,m,i)}\right]=0,
\end{array}\right\}
\end{equation}
where we apply the Lenglart inequality to derive the convergence of the second term in the right side of \eqref{tnmi} from the second convergence of \eqref{aim3o} (for this application we need to introduce the stopping time $T(n,m,i)$, which enables us to drop the indicator $1_{\Omega_n(T,m)}$). 
Recall that $|b(m)|\leq(1+m)A$ and $|\delta(s,z)|\leq\gamma(z)\leq A$ due to $(\ref{SH})$, so we have for $(i-1)\Delta_n\leq u<T(n,m,i)$ (then $|\Xi(m,i)_u|\leq 5A$):
\begin{align*}
&\left|a(n,m,i)_u\right|\lesssim m|\Xi(m,i)_u|^2+\overline{\gamma}_m|\Xi(m,i)_u|,\\
&\textstyle\int_{A_m^c}k(\Xi(m,i)_{u},\delta(u,z))^2\lambda(dz)\lesssim \overline{\gamma}_m|\Xi(m,i)_u|^2
\end{align*}
by $(\ref{SH})$, $(\ref{g.est})$--$(\ref{k.est})$ and $(\ref{phi.est})$. Combining these estimates with $(\ref{jump.est})$ as well as the fact that $\overline{\gamma}_m\to0$ as $m\to\infty$, we conclude that $(\ref{aim3o})$ holds true.\hfill$\Box$

\subsection{Proof of Proposition \ref{prop:counter-ex}}\label{proof:counter-ex}

We can easily check that $g_a$ is a $C^2$ function and we have
\[
g_a'(x)=x|x|\left\{3\sin(2a\log|x|)+2a\cos(2a\log|x|)\right\} 
\]
and
\[
g_a''(x)=|x|\left\{6\sin(2a\log|x|)+10a\cos(2a\log|x|)
-4a^2\sin(2a\log|x|)\right\}
\]
for any $x\in\mathbb{R}$. Hence claim (a) holds true.  

Next we prove claim (b). 
In the following we denote by $\mathfrak{N}$ the standard normal density. 
First we show that there is a real number $a$ such that 
\begin{equation}\label{nonzero}
\int_0^\infty x^3\sin(2a\log x)\mathfrak{N}(x)dx\neq0.
\end{equation}
In fact, substituting $y=\log x$, we have
\[
\int_0^\infty x^3\sin(2a\log x)\mathfrak{N}(x)dx
=\int_{-\infty}^\infty e^{4y}\mathfrak{N}(e^y)\sin(2ay)dy.
\]
Since the function $\mathbb{R}\ni y\mapsto e^{4y}\mathfrak{N}(e^y)\in\mathbb{R}$ is square integrable and not even, the imaginary part of its Fourier transform is not identical to zero. Hence \eqref{nonzero} holds true for some $a\in\mathbb{R}$. 

Now we show that the variables $\mathcal{V}^n_T(X^c,g_a)/\sqrt{\Delta_n}$ do not converge in law with $\Delta_n=\exp(-n\pi/a)$ if $a$ satisfies \eqref{nonzero} (note that such an $a$ must not be zero). 
To obtain a contradiction, suppose that the variables $\mathcal{V}^n_T(X^c,g_a)/\sqrt{\Delta_n}$ converge in law to some random variable $Z$ as $n\to\infty$. 
Since we have
\begin{align*}
\variance\left[\frac{1}{\sqrt{\Delta_n}}\mathcal{V}^n_T(X^c,g_a)\right]
=\frac{1}{\Delta_n}\sum_{i=1}^{\lfloor T/\Delta_n\rfloor}\variance\left[g_a(\Delta^n_iX^c)\right]
\leq \frac{1}{\Delta_n}\sum_{i=1}^{\lfloor T/\Delta_n\rfloor}E\left[(\Delta^n_iX^c)^6\right]
=O(\Delta_n)
\end{align*}
and
\begin{align*}
\frac{1}{\sqrt{\Delta_n}}|E[\mathcal{V}^n_T(X^c,g_a)]|\leq
\frac{1}{\sqrt{\Delta_n}}\sum_{i=1}^{\lfloor T/\Delta_n\rfloor}E\left[|\Delta^n_iX^c|^3\right]
=O(1),
\end{align*}
we obtain
\[
\sup_{n\in\mathbb{N}}E\left[\left|\frac{1}{\sqrt{\Delta_n}}\mathcal{V}^n_T(X^c,g_a)\right|^2\right]<\infty.
\]
Therefore, the variables $\mathcal{V}^n_T(X^c,g_a)/\sqrt{\Delta_n}$ are uniformly integrable, and thus Theorem 3.5 of \cite{Billingsley1999} implies that $Z$ is integrable and 
\begin{equation}\label{contradiction}
E\left[\frac{1}{\sqrt{\Delta_n}}\mathcal{V}^n_T(X^c,g_a)\right]\to E[Z]
\end{equation}
as $n\to\infty$. 
In the meantime, we have
\begin{align*}
E\left[\frac{1}{\sqrt{\Delta_n}}\mathcal{V}^n_T(X^c,g_a)\right]
&=\frac{1}{\sqrt{\Delta_n}}\sum_{i=1}^{\lfloor T/\Delta_n\rfloor}E\left[g_a(\Delta^n_iX^c)\right]
=\frac{2T}{\Delta_n^{3/2}}\int_0^\infty g_a(\sqrt{\Delta_n}x)\mathfrak{N}(x)dx+O(\Delta_n)\\
&=2T\int_0^\infty x^3\sin\left(2a\log(\sqrt{\Delta_n}x)\right)\mathfrak{N}(x)dx+O(\Delta_n)
\end{align*}
as $n\to\infty$. Hence, in view of \eqref{contradiction}, the sequence
\[
c_n:=\int_0^\infty x^3\sin\left(2a\log(\sqrt{\Delta_n}x)\right)\mathfrak{N}(x)dx,\qquad n=1,2,\dots
\]
converges as $n\to\infty$. Using the identity
\begin{align*}
\sin\left(2a\log(\sqrt{\Delta_n}x)\right)
&=\sin(a\log\Delta_n)\cos\left(2a\log x\right)
+\cos(a\log\Delta_n)\sin\left(2a\log x\right),
\end{align*}
we can rewrite $c_n$ as
\begin{align*}
c_n
&=\sin(a\log\Delta_n)\int_0^\infty x^3\cos\left(2a\log x\right)\mathfrak{N}(x)dx
+\cos(a\log\Delta_n)\int_0^\infty x^3\sin\left(2a\log x\right)\mathfrak{N}(x)dx.
\end{align*}
Since $\Delta_n=\exp(-n\pi/a)$, we obtain
\[
c_n=(-1)^n\int_0^\infty x^3\sin\left(2a\log x\right)\mathfrak{N}(x)dx.
\]
Therefore, the sequence $c_n$ does not converge due to \eqref{nonzero}, a contradiction.\hfill$\Box$

\subsection{Proof of Theorem \ref{mainthm}}\label{proof:mainthm}

\subsubsection{Localization}

As in Section \ref{proof:thm2}, a standard localization argument allows us to replace the assumptions [A2]--[A3] by the following strengthened versions:
\begin{enumerate}[label={\normalfont[SA\arabic*]},start=2]

\item \label{hypo:SA2}
We have \ref{hypo:A2}, and the processes $X_t$, $b_t$, $\sigma_t$, $\widetilde{b}_t$ and $\widetilde{\sigma}_t$ are bounded. Moreover, there are a constant $\Lambda$ and a non-negative bounded function $\gamma$ on $E$ such that $\int\gamma(z)^2\lambda(\mathrm{d}z)<\infty$ and $|\delta(\omega,t,z)|\vee|\widetilde{\delta}(\omega,t,z)|\leq\gamma(z)$ and
\begin{align*}
&E\left[|b_{t_1}-b_{t_2}|^2|\mathcal{F}_{t_1\wedge t_2}\right]
\leq \Lambda E\left[|t_1-t_2|^\varpi|\mathcal{F}_{t_1\wedge t_2}\right],\\
&E\left[|\delta(t_1,z)-\delta(t_2,z)|^2|\mathcal{F}_{t_1\wedge t_2}\right]
\leq \Lambda\gamma(z)^2E\left[|t_1-t_2|^\varpi|\mathcal{F}_{t_1\wedge t_2}\right]
\end{align*}
for any bounded $(\mathcal{F}_t)$-stopping times $t_1$ and $t_2$.

\item \label{hypo:SA3}
The process $\int u^6 Q_t(\mathrm{d}z)$ is bounded and there is a constant $\Lambda'$ such that
\[
E\left[|\alpha_{t_1}-\alpha_{t_2}|^2|\mathcal{F}_{t_1\wedge t_2}\right]
\leq \Lambda' E\left[|t_1-t_2|^\varpi|\mathcal{F}_{t_1\wedge t_2}\right]
\]
for any bounded $(\mathcal{F}_t)$-stopping times $t_1$ and $t_2$. Moreover, $\alpha_t$ is c\`adl\`ag.

\end{enumerate}

In the following we fix a constant $\xi\in(0,1)$ such that
\begin{equation}\label{est.xi}
\xi>\frac{11}{12}\vee\frac{2+\varpi}{2(1+\varpi)},
\end{equation}
and we set $\bar{r}_n=\Delta_n^{\xi}$. By a similar argument to Section 6.1.1 of \cite{Koike2015jclt}, we can further replace the assumption [A1] by the following strengthened version:
\begin{enumerate}[label={\normalfont[SA1]}]

\item \label{hypo:SA1} We have \ref{hypo:A1}, and for every $n$ it holds that
\begin{equation}\label{SA4}
\sup_{i\geq0}(t^n_i-t^n_{i-1})\leq\bar{r}_n.
\end{equation}

\end{enumerate}

\subsubsection{Notation and estimates}
We use the same notation as in Section \ref{proof:thm2} with the following change for the definition of the set $\Omega_n(T,m)$: For $m,n\in\mathbb{N}$, we denote by $\Omega_n(T,m)$ the set on which $k_n-1\leq N^n_{S_p-}\leq N^n_T-k_n$ for all $p\in\mathcal{P}_m$ such that $S_p\leq T$ and $|S_{p_1}-S_{p_2}|>k_n\bar{r}_n$ for any $p_1,p_2\in\mathcal{P}_m$ such that $p_1\neq p_2$ and $S_{p_1},S_{p_2}<\infty$. We have $\lim_nP(\Omega_n(T,m))=1$.
We additionally define the processes $B(m)$ and $Z(m)$ by $B(m)_t=\int_0^tb(m)_sds$ and $Z(m)_t=\int_0^t\int_{A_m^c}\delta(s,z)(\mu-\nu)(ds,dz)$, respectively.  
We also define the $\mathbb{R}^2\otimes\mathbb{R}^2$-valued variable $\overline{\Gamma}(m)_T$ by
\begin{align*}
\overline{\Gamma}(m)^{11}_T&=\frac{4}{\psi_2^{2}}\sum_{p\in\mathcal{P}_m,S_p\leq T}(\Delta X_{S_p})^2\left\{\Phi_{22}\theta\left(\sigma_{S_p}^2G_{S_p}+\sigma_{s-}^2G_{S_p-}\right)+\frac{\Phi_{12}}{\theta}\left(\alpha_s+\alpha_{s-}\right)\right\},\\
\overline{\Gamma}(m)^{12}_T&=\frac{6}{\psi_2\psi_3}\sum_{p\in\mathcal{P}_m,S_p\leq T}(\Delta X_{S_p})^3\left\{\theta\left(\Phi_{23+}\sigma_{S_p}^2G_{S_p}+\Phi_{23-}\sigma_{S_p-}^2G_{S_p-}\right)+\theta^{-1}\left(\Phi'_{23+}\alpha_{S_p}+\Phi'_{23-}\alpha_{S_p-}\right)\right\},\\
\overline{\Gamma}(m)_T^{22}&=\frac{9}{\psi_3^{2}}\sum_{p\in\mathcal{P}_m,S_p\leq T}(\Delta X_{S_p})^4\left\{\theta\left(\Phi_{3+}\sigma_{S_p}^2G_{S_p}+\Phi_{3-}\sigma_{S_p-}^2G_{S_p-}\right)+\theta^{-1}\left(\Phi'_{3+}\alpha_{S_p}+\Phi'_{3-}\alpha_{S_p-}\right)\right\}.
\end{align*}

We set $g^n_p=g(p/k_n)$ for $p=0,1,\dots,k_n$ and $\Delta(g)^n_p=g^n_{p+1}-g^n_p$ for $p=0,1,\dots,k_n-1$.
We also set $I_i=[t^n_{i-1},t^n_i)$ and $\overline{I}_i=[t^n_{i-1},t^n_{i+k_n-1})$ for $i=0,1,\dots$.

For every $i\geq0$ we define the process $\overline{g}^n_i$ by $\overline{g}^n_i(s)=\sum_{p=1}^{k_n-1}g^n_p1_{I_{i+p}}(s)$. For any semimartingale $V$, we define the process $\overline{V}_{i,t}$ by $\overline{V}_{i,t}=\int_0^t\overline{g}^n_i(s-)dV_s$. Note that $\overline{V}_{i}=\overline{V}_{i,t^n_{i+k_n-1}}$.

Recall that, for random variables $X$ and $Y$ which may depend on the parameters $n,m,i$, $X\lesssim Y$ means that there exists a (non-random) constant $K>0$ independent of $n,m,i$ such that $X\leq KY$ a.s. In addition, if $K$ possibly depends on $m$, we write $X\lesssim_m Y$ instead.

\ref{hypo:SA2} and \eqref{SA4} yield
\begin{equation}\label{moment:jump}
E\left[\sup_{s\in\overline{I}_i}\left|\overline{X'(m)}_{i,s}\right||\mathcal{F}_{t^n_{i-1}}\right]
\lesssim mE\left[|\overline{I}_i||\mathcal{F}_{t^n_{i-1}}\right]\lesssim mk_n\bar{r}_n
\end{equation}
and
\begin{equation}\label{moment:drift}
\left|\overline{B(m)}_i\right|
\lesssim m\left|\overline{I}_i\right|\lesssim mk_n\bar{r}_n.
\end{equation}
Here, $|\cdot|$ denotes the Lebesgue measure. The BDG inequality, \ref{hypo:SA2} and \eqref{SA4} yield
\begin{equation}\label{moment:C}
E\left[\sup_{s\in\overline{I}_i}\left|\overline{X}^c_{i,s}\right|^r|\mathcal{F}_{t^n_{i-1}}\right]
\lesssim E\left[|\overline{I}_i|^{r/2}|\mathcal{F}_{t^n_{i-1}}\right]\lesssim (k_n\bar{r}_n)^{r/2}
\qquad\text{for any }r>0
\end{equation}
and
\begin{equation}\label{moment:Z}
E\left[\sup_{s\in\overline{I}_i}\left|\overline{Z(m)}_{i,s}\right|^2|\mathcal{F}_{t^n_{i-1}}\right]
\lesssim \overline{\gamma}_mE\left[|\overline{I}_i||\mathcal{F}_{t^n_{i-1}}\right]\lesssim \overline{\gamma}_mk_n\bar{r}_n,
\end{equation}
where $\overline{\gamma}_m=\int_{A_m^c}\gamma(z)^2\lambda(dz)$. Note that $\overline{\gamma}_m\to0$ as $m\to\infty$ by the dominated convergence theorem. The BDG inequality and \ref{hypo:SA3} yield
\begin{equation}\label{moment:noise}
E\left[\left|\overline{\epsilon}_i\right|^r|\mathcal{F}\right]
+E\left[\left|\overline{\epsilon}_i\right|^r|\mathcal{F}^n_{t^n_{i-1}}\right]
\leq K_rk_n^{-r/2}\qquad\text{for any }r\in[2,6].
\end{equation}

Finally, Lemma 6.1 of \cite{Koike2015} implies that
\begin{equation}\label{C3}
N^n_T=O_p(\Delta_n^{-1})
\end{equation}
as $n\to\infty$ for every $t>0$.

\subsubsection{Main body}

For each $m\in\mathbb{N}$, we consider the following decomposition of $\Delta_n^{-1/4}\left(\PCV^n_T-\sum_{0\leq s\leq T}(\Delta X_s)^3\right)$:
\begin{align*}
&\Delta_n^{-1/4}\left(\PCV^n_T-\sum_{0\leq s\leq T}(\Delta X_s)^3\right)\\
&=\frac{\Delta_n^{-1/4}}{\psi_3k_n}\sum_{i=0}^{N^n_T-k_n+1}\left\{\left(\overline{X(m)}_i+\overline{\epsilon}_i\right)^3-\left(\overline{X(m)}_i\right)^3\right\}
+\Delta_n^{-1/4}\left\{\frac{1}{\psi_3k_n}\sum_{i=0}^{N^n_T-k_n+1}\left(\overline{X(m)}_i\right)^3-\sum_{0\leq s\leq T}(\Delta X(m)_s)^3\right\}\\
&\qquad+3\frac{\Delta_n^{-1/4}}{\psi_3k_n}\sum_{i=0}^{N^n_T-k_n+1}\left(\overline{X(m)}_i+\overline{\epsilon}_i\right)^2\overline{X'(m)}_i
+3\frac{\Delta_n^{-1/4}}{\psi_3k_n}\sum_{i=0}^{N^n_T-k_n+1}\left(\overline{C(m)}_i+\overline{\epsilon}_i\right)\left(\overline{X'(m)}_i\right)^2\\
&\qquad+3\frac{\Delta_n^{-1/4}}{\psi_3k_n}\sum_{i=0}^{N^n_T-k_n+1}\overline{Z(m)}_i\left(\overline{X'(m)}_i\right)^2
+\Delta_n^{-1/4}\left\{\frac{1}{\psi_3k_n}\sum_{i=0}^{N^n_T-k_n+1}\left(\overline{X'(m)}_i\right)^3-\sum_{0\leq s\leq T}(\Delta X'(m)_s)^3\right\}\\
&=:\mathbb{I}_n(m)+\mathbb{II}_n(m)+\mathbb{III}_n(m)+\mathbb{IV}_n(m)+\mathbb{V}_n(m)+\mathbb{VI}_n(m).
\end{align*}
Since we have
\[
\lim_{m\to\infty}\limsup_{n\to\infty}P_n\left(\Delta_n^{-1/4}\left|\left(\PRV^n_T-[X,X]_T\right)-\left(\PRV(m)^n_T-[Y(m),Y(m)]_T\right)\right|>\eta\right)=0
\]
for any $\eta>0$ by Proposition 6.3 from \cite{Koike2015jclt}, where
\[
\PRV(m)^n_T=\frac{1}{\psi_2k_n}\sum_{i=0}^{N^n_T-k_n+1}\left(\overline{Y(m)}_i+\overline{\epsilon}_i\right)^2-\frac{\psi_1}{2\psi_2k_n}\sum_{i=1}^{N^n_T}(Y_{t^n_i}-Y_{t^n_{i-1}})^2,
\]
and we obviously have $\overline{\Gamma}(m)_T\to^P\overline{\Gamma}_T$ as $m\to\infty$, the proof is completed once we show the following convergences for any $\eta>0$ due to Proposition 2.2.4 of \cite{JP2012}:
\begin{align}
&\lim_{m\to\infty}\limsup_{n\to\infty}P_n\left(\left|\mathbb{I}_n(m)\right|>\eta\right)=0,\label{aim1}\\
&\lim_{m\to\infty}\limsup_{n\to\infty}P_n\left(\left|\mathbb{II}_n(m)\right|>\eta\right)=0,\label{aim2}\\
&\mathbb{III}_n(m)\to^P0\qquad\text{for any }m\in\mathbb{N},\label{aim3}\\
&\left(\Delta_n^{-1/4}\left(\PRV(m)^n_T-[Y(m),Y(m)]_T\right),\mathbb{IV}_n(m)\right)^*
\to^\mathcal{S}\Gamma(m)_T^{1/2}\zeta,\label{aim4}\\
&\lim_{m\to\infty}\limsup_{n\to\infty}P_n\left(\left|\mathbb{V}_n(m)\right|>\eta\right)=0,\label{aim5}\\
&\mathbb{VI}_n(m)\to^P0\qquad\text{for any }m\in\mathbb{N},\label{aim6}
\end{align}
where $\zeta$ is a bivariate standard normal variable which is defined on an extension of $\mathcal{B}$ and independent of $\mathcal{F}$, and
\[
\Gamma(m)_T=
\left[
\begin{array}{cc}
\Gamma^c_T+\overline{\Gamma}(m)_T^{11}  & \overline{\Gamma}(m)_T^{12}  \\
\overline{\Gamma}(m)_T^{12}  & \overline{\Gamma}(m)_T^{22}
\end{array}
\right].
\]

\begin{proof}[Proof of \eqref{aim6}]
On the set $\Omega_n(T,m)$ we have
\begin{align*}
\mathbb{VI}_n(m)
=\frac{\Delta_n^{-1/4}}{\psi_3}\sum_{p\in\mathcal{P}_m:S_p\leq T}\left\{\left(\frac{1}{k_n}\sum_{i=1}^{k_n-1}\left(g^n_i\right)^3\right)-\psi_3\right\}\left(\Delta X_{S_p}\right)^3.
\end{align*}
Since $\frac{1}{k_n}\sum_{i=1}^{k_n-1}\left(g^n_i\right)^3=\psi_3+O(k_n^{-1})$ by the Lipschitz continuity of $g$ and $\lim_nP(\Omega_{n}(T,m))=1$, we obtain the desired result.
\end{proof}

\begin{proof}[Proof of \eqref{aim1}]
We decompose the target quantity as
\begin{align*}
\mathbb{I}_n(m)
&=\frac{\Delta_n^{-1/4}}{\psi_3k_n}\sum_{i=0}^{N^n_{T}-k_n+1}\left[3\left(\overline{Z(m)}_i\right)^2\overline{\epsilon}_i+\left\{\left(\overline{X(m)}_i+\overline{\epsilon}_i\right)^3-\left(\overline{X(m)}_i\right)^3-3\left(\overline{Z(m)}_i\right)^2\overline{\epsilon}_i\right\}\right]\\
&=:\mathbb{I}_n^{(1)}(m)+\mathbb{I}_n^{(2)}(m).
\end{align*}
It suffices to prove
\begin{equation}\label{aimI}
\lim_{m\to\infty}\limsup_{n\to\infty}P_n\left(\left|\mathbb{I}_n^{(l)}(m)\right|>\eta\right)=0
\end{equation}
for $l=1,2$. First, by \eqref{moment:noise} and \eqref{moment:Z} we have
\begin{align*}
E\left[\left|\mathbb{I}_n^{(1)}(m)\right|\right]
\lesssim\frac{1}{k_n}E\left[\sum_{i=0}^{N^n_{T}+1}\left(\overline{Z(m)}_i\right)^2\right]
\lesssim\frac{\overline{\gamma}_m}{k_n}E\left[\sum_{i=0}^{N^n_{T}+1}\left|\overline{I}_i\right|\right]
\lesssim\overline{\gamma}_m,
\end{align*}
where we use the following inequality to obtain the final upper bound:
\begin{align*}
\sum_{i=0}^{N^n_{T}+1}\left|\overline{I}_i\right|
&=\sum_{i=0}^{N^n_{T}+1}\sum_{p=0}^{k_n-1}|I_{i+p}|
=\sum_{p=0}^{k_n-1}\sum_{i=0}^{N^n_T-k_n+1}|I_{i+p}|
+\sum_{i=N^n_T-k_n+2}^{N^n_{T}+1}\sum_{p=0}^{k_n-1}|I_{i+p}|\\
&\leq k_nT+k_n(k_n-1)\bar{r}_n
\lesssim k_n.
\end{align*}
Hence \eqref{aimI} holds true for $l=1$.

Next we consider the case $l=2$, and we start with some preliminary results. First we note that
\begin{equation}\label{Xedge}
\sup_{N^n_T-k_n+1<i\leq N^n_T+1}\left|\overline{V}_i\right|=O_p(\sqrt{k_n\bar{r}_n})
\end{equation}
for $V\in\{X(m),Z(m)\}$. In fact, summation by parts yields $\overline{V}_i=-\sum_{p=0}^{k_n-1}\Delta(g)^n_p(V_{t^n_{i+p}}-V_{t^n_i})$, and $\sup_{|h|\leq h_0}|V_{t+h}-V_t|=O_p(\sqrt{h_0})$ as $h_0\downarrow0$ by \ref{hypo:SA2} and the Doob inequality, hence \eqref{Xedge} holds ture by \eqref{SA4}.
Next, for any $K>0$ we define the $(\mathcal{F}_t)$-stopping time $R^n_K$ by
\begin{equation}\label{def:Rn}
R^n_{K}=\inf\{s:n^{-1}N^n_s>K\}.
\end{equation}
Since $\Delta N^n_s\leq 1$ for every $s$, it holds that
\begin{equation}\label{SC3}
N^n_{s\wedge R^n_K}\leq K n+1
\end{equation}
for all $s\geq0$. Moreover, by \eqref{C3} we also have
\begin{equation}\label{alpha.infinity}
\limsup_{K\to\infty}\limsup_{n\to\infty}P\left(R^n_K\leq t\right)=0.
\end{equation}

Now we turn to the proof of \eqref{aimI} for the case $l=2$. For each $m$, set
\[
\zeta(m)^n_i=\frac{\Delta_n^{-1/4}}{\psi_3k_n}\left\{\left(\overline{X(m)}_i+\overline{\epsilon}_i\right)^3-\left(\overline{X(m)}_i\right)^3-3\left(\overline{Z(m)}_i\right)^2\overline{\epsilon}_i\right\},\qquad
i=0,1,\dots.
\]
Then, by \eqref{moment:noise} and \eqref{Xedge} we have $\mathbb{I}_n^{(2)}(m)=\sum_{i=0}^{N^n_T+1}\zeta(m)^n_i+o_p(1)$ as $n\to\infty$. Therefore, by the Markov inequality and \eqref{alpha.infinity} it is enough to prove $\sum_{i=0}^{N^n_{T\wedge R^n_K}+1}\zeta(m)^n_i=o_p(1)$ as $n\to\infty$ for any fixed $K>0$ and $m\in\mathbb{N}$.

Since integration by parts yields $\overline{X}^c_i\overline{Z(m)}_i=\int_{\overline{I}_i}\overline{X}^c_{i,s-}d\overline{Z(m)}_{i,s}+\int_{\overline{I}_i}\overline{Z(m)}_{i,s-}d\overline{X}^c_{i,s}$, we have
\begin{equation}\label{revise-ineq}
E\left[\left|\overline{X}^c_i\overline{Z(m)}_i\right|^2|\mathcal{F}_{t^n_{i-1}}\right]
\lesssim E\left[\left(\sup_{s\in\overline{I}_i}\left|\overline{X}^c_{i,s-}\right|^2+\sup_{s\in\overline{I}_i}\left|\overline{Z(m)}_{i,s-}\right|^2\right)\left|\overline{I}_i\right|\right]
\lesssim\left(k_n\bar{r}_n\right)^2
\end{equation}
by \ref{hypo:SA2}, \eqref{SA4} and \eqref{moment:C}--\eqref{moment:Z}. 
Moreover, we can rewrite $\zeta(m)^n_i$ as
\[
\zeta(m)^n_i=\frac{\Delta_n^{-1/4}}{\psi_3k_n}\left\{3\left(\left(\overline{C(m)}_i\right)^2+2\overline{C(m)}_i\overline{Z(m)}_i\right)\overline{\epsilon}_i+3\overline{X(m)}_i\left(\overline{\epsilon}_i\right)^2+\left(\overline{\epsilon}_i\right)^3\right\}.
\]
Hence, using the relation $C(m)_t=X_0+B(m)_t+X^c_t$ and estimates \eqref{moment:drift}--\eqref{moment:noise} and \eqref{revise-ineq}, we obtain $E[|\zeta(m)^n_i|^2|\mathcal{F}^n_{t^n_{i-1}}]\lesssim_m\bar{r}_n^2$. 
Therefore, noting that $\zeta(m)^n_i$ is $\mathcal{F}^n_{t^n_{i+k_n-1}}$-measurable, we have
\begin{align*}
E\left[\left|\sum_{i=0}^{N^n_{T\wedge R^n_K}+1}\left(\zeta(m)^n_i-E[\zeta(m)^n_i|\mathcal{F}^n_{t^n_{i-1}}]\right)\right|^2\right]
\lesssim_m \Delta_n^{-1}k_n\bar{r}_n^2=o(1).
\end{align*}
Hence it holds that $\sum_{i=0}^{N^n_{T\wedge R^n_K}+1}\zeta(m)^n_i=\sum_{i=0}^{N^n_{T\wedge R^n_K}+1}E[\zeta(m)^n_i|\mathcal{F}^n_{t^n_{i-1}}]+o_p(1)$. Now, since $E[\overline{\epsilon}_i|\mathcal{F}]=0$, we can decompose $\sum_{i=0}^{N^n_{T\wedge R^n_K}+1}E[\zeta(m)^n_i|\mathcal{F}^n_{t^n_{i-1}}]$ as
\begin{align*}
\sum_{i=0}^{N^n_{T\wedge R^n_K}+1}E[\zeta(m)^n_i|\mathcal{F}^n_{t^n_{i-1}}]
&=\frac{\Delta_n^{-1/4}}{\psi_3k_n}\sum_{i=0}^{N^n_{T\wedge R^n_K}+1}\left\{3E[\overline{X(m)}_i\left(\overline{\epsilon}_i\right)^2|\mathcal{F}^n_{t^n_{i-1}}]+E[\left(\overline{\epsilon}_i\right)^3|\mathcal{F}^n_{t^n_{i-1}}]\right\}\\
&=:\mathbb{A}_{1,n}+\mathbb{A}_{2,n}
\end{align*}
(we drop the index $m$ because we fix it here). First we consider $\mathbb{A}_{1,n}$. We can rewrite it as
\[
\mathbb{A}_{1,n}=\frac{3\Delta_n^{-1/4}}{\psi_3k_n}\sum_{i=0}^{N^n_{T\wedge R^n_K}+1}\sum_{p=0}^{k_n-1}(\Delta(g)^n_p)^2E[\overline{X(m)}_i\alpha_{t^n_{i+p}}|\mathcal{F}^n_{t^n_{i-1}}].
\]
Since we have $E[\overline{X(m)}_i\alpha_{t^n_{i-1}}|\mathcal{F}_{t^n_{i-1}}]=E[\overline{B(m)}_i\alpha_{t^n_{i-1}}|\mathcal{F}_{t^n_{i-1}}]$, it holds that
\begin{align*}
\left|E[\overline{X(m)}_i\alpha_{t^n_{i+p}}|\mathcal{F}_{t^n_{i-1}}]\right|
&\leq E\left[\left|\overline{X(m)}_i(\alpha_{t^n_{i+p}}-\alpha_{t^n_{i-1}})\right||\mathcal{F}_{t^n_{i-1}}\right]+E\left[\left|\overline{B(m)}_i\alpha_{t^n_{i-1}}\right||\mathcal{F}_{t^n_{i-1}}\right]\\
&\lesssim_m(k_n\bar{r}_n)^{(1+\varpi)/2}+k_n\bar{r}_n
\end{align*}
by the Schwarz inequality, \ref{hypo:SA3} and \eqref{moment:drift}--\eqref{moment:Z}. Therefore, we obtain $\mathbb{A}_{1,n}=O_p\left(\Delta_n^{-\frac{1}{4}+\left(\frac{1+\varpi}{2}\wedge1\right)\left(\xi-\frac{1}{2}\right)}\right)=o_p(1)$ by \eqref{est.xi} after distinguishing the cases $\varpi\geq 1$ and $\varpi < 1$. Next, let us consider $\mathbb{A}_{2,n}$. For any nonnegative integers $p,q,r$, $E[\epsilon_{t^n_{i+p}}\epsilon_{t^n_{i+q}}\epsilon_{t^n_{i+r}}|\mathcal{F}^n_{t^n_{i-1}}]$ does not vanish only if $p=q=r$, hence we have
\begin{align*}
\mathbb{A}_{2,n}
&=\frac{\Delta_n^{-1/4}}{\psi_3k_n}\sum_{i=0}^{N^n_{T\wedge R^n_K}+1}\sum_{p=0}^{k_n-1}(\Delta(g)^n_p)^3E[(\epsilon_{t^n_{i+p}})^3|\mathcal{F}^n_{t^n_{i-1}}]
=O_p(\Delta_n^{-1/4}k_n^{-1}\cdot \Delta_n^{-1}k_n^{-2})=o_p(1).
\end{align*}
Consequently, we conclude that $\sum_{i=0}^{N^n_{T\wedge R^n_K}+1}\zeta(m)^n_i=o_p(1)$ and the proof is completed.
\end{proof}

\begin{proof}[Proof of \eqref{aim4}]
The proof is analogous to that of Proposition 6.2 of \cite{Koike2015jclt}, which is based on Propositions 6.4--6.7 of that paper. So we omit it.
\end{proof}

Recall that, for a locally square-integrable martingale $M$ such that $M_0=0$, $\langle M\rangle$ denotes the predictable quadratic variation of $M$, i.e.~the predictable increasing process such that $M^2-\langle M\rangle$ is a local martingale (such a process always exists and is unique; see e.g.~Theorem 4.2 from Chapter I of \cite{JS2003}).    
The next inequality plays a key role in the remaining proof:
\begin{lem}\label{lenglart2}
We have
\begin{align*}
E\left[\sup_{\tau_1\leq t\leq\tau_2}\left|M_t\right||\mathcal{F}_{\tau_1}\right]\leq
3E\left[\sqrt{\langle M\rangle_{\tau_2}}|\mathcal{F}_{\tau_1}\right]
\end{align*}
for any stopping time $\tau_1,\tau_2$ such that $\tau_1\leq\tau_2$ and for any locally square-integrable martingale $M$ such that $M_0=0$.
\end{lem}

\begin{proof}
This result is a direct consequence of Theorem 5 from Chapter 1, Section 9 of \cite{LS1989}.
\end{proof}

\begin{proof}[Proof of \eqref{aim3}]
Set $L(m)=X^c+Z(m)$. Then we have
\begin{align*}
\left|\mathbb{III}_n(m)\right|
&\leq3\frac{\Delta_n^{-1/4}}{\psi_3k_n}\sum_{i=0}^{N^n_T-k_n+1}\left\{\left(\overline{L(m)}_i\right)^2+\left(\overline{B(m)}_i+\overline{\epsilon}_i\right)^2\right\}\left|\overline{X'(m)}_i\right|\\
&=:\mathbb{III}_n^{(1)}+\mathbb{III}^{(2)}_n,
\end{align*}
so it suffices to prove $\mathbb{III}_n^{(l)}\to^P0$ as $n\to\infty$ for $l=1,2$ (note that we drop the index $m$ because it is fixed in this part). First, \eqref{moment:drift}, \eqref{moment:noise} and \eqref{moment:jump} yield
\begin{align*}
E\left[\left|\mathbb{III}_n^{(2)}\right|\right]
\lesssim_m\frac{\Delta_n^{-1/4}}{k_n^2}E\left[\sum_{i=0}^{N^n_T+1}\left|\overline{X'(m)}_i\right|\right]
\lesssim_m \Delta_n^{-1/4}k_n^{-1}=o(1),
\end{align*}
hence we have $\mathbb{III}^{(2)}_n\to^P0$ as $n\to\infty$.

To prove $\mathbb{III}^{(1)}_n\to^P0$ as $n\to\infty$, it suffices to show that there is a constant $K$ (which may depend on $m$) such that
\begin{equation}\label{aimIII}
E\left[\left|\left(\overline{L(m)}_i\right)^2\overline{X'(m)}_i\right||\mathcal{F}_{t^n_{i-1}}\right]
\leq K(k_n\bar{r}_n)^2
\end{equation}
for any $i,n$ because of the Lenglart inequality, \eqref{C3} and the fact that $\xi>7/8$. To prove \eqref{aimIII}, we consider the following decomposition of $\left(\overline{L(m)}_i\right)^2\overline{X'(m)}_i$, which is obtained by applying integration by parts repeatedly (note that $[L(m),X'(m)]\equiv0$ by construction):
\begin{align*}
\left(\overline{L(m)}_i\right)^2\overline{X'(m)}_i
&=\int_{\overline{I}_i}\left(\overline{L(m)}_{i,s-}\right)^2d\overline{X'(m)}_{i,s}
+2\int_{\overline{I}_i}\overline{X'(m)}_{i,s-}\overline{L(m)}_{i,s-}d\overline{L(m)}_{i,s}\\
&+\int_{\overline{I}_i}\overline{X'(m)}_{i,s-}d[\overline{L(m)}_{i,\cdot}]_s\\
&=:\mathbf{A}^{n,1}_i+2\mathbf{A}^{n,2}_i+\mathbf{A}^{n,3}_i.
\end{align*}
Then, it is enough to show that
\begin{equation}\label{aimIII-2}
E\left[\left|\mathbf{A}^{n,l}_i\right||\mathcal{F}_{t^n_{i-1}}\right]
\lesssim_m (k_n\bar{r}_n)^2
\end{equation}
for every $l=1,2,3$. First, we have
\begin{align*}
E\left[\left|\mathbf{A}^{n,1}_i\right||\mathcal{F}_{t^n_{i-1}}\right]
&\lesssim_mE\left[\sup_{s\in\overline{I}_i}\left(\overline{L(m)}_{i,s}\right)^2|\overline{I}_i|\big|\mathcal{F}_{t^n_{i-1}}\right]\\
&\lesssim_mk_n\bar{r}_nE\left[\sup_{s\in\overline{I}_i}\left(\overline{L(m)}_{i,s}\right)^2|\mathcal{F}_{t^n_{i-1}}\right]
\lesssim_m\left(k_n\bar{r}_n\right)^2
\end{align*}
by \ref{hypo:SA2}, \eqref{SA4} and \eqref{moment:C}--\eqref{moment:Z}, so \eqref{aimIII-2} holds true for $l=1$. Next, Lemma \ref{lenglart2} and \eqref{SA4} yield
\begin{align*}
E\left[\left|\mathbf{A}^{n,2}_i\right||\mathcal{F}_{t^n_{i-1}}\right]
&\leq 3E\left[\left\{\int_{\overline{I}_i}\left|\overline{X'(m)}_{i,s-}\overline{L(m)}_{i,s-}\right|^2d\langle\overline{L(m)}_{i,\cdot}\rangle_s\right\}^{1/2}|\mathcal{F}_{t^n_{i-1}}\right]\\
&\lesssim E\left[\sup_{s\in\overline{I}_i}\left|\overline{X'(m)}_{i,s}\overline{L(m)}_{i,s}\right|\sqrt{|\overline{I}_i|}|\mathcal{F}_{t^n_{i-1}}\right]\\
&\lesssim \sqrt{k_n\bar{r}_n}E\left[\sup_{s\in\overline{I}_i}\left|\overline{X'(m)}_{i,s}\overline{L(m)}_{i,s}\right||\mathcal{F}_{t^n_{i-1}}\right].
\end{align*}
Now, integration by parts, Lemma \ref{lenglart2}, \ref{hypo:SA2}, \eqref{SA4}, \eqref{moment:jump} and \eqref{moment:C}--\eqref{moment:Z} imply that
\begin{align*}
&E\left[\sup_{s\in\overline{I}_i}\left|\overline{X'(m)}_{i,s}\overline{L(m)}_{i,s}\right||\mathcal{F}_{t^n_{i-1}}\right]\\
&\lesssim_mE\left[\sup_{s\in\overline{I}_i}\left|\overline{X'(m)}_{i,s}\right|\sqrt{\left|\overline{I}_i\right|}|\mathcal{F}_{t^n_{i-1}}\right]
+E\left[\sup_{s\in\overline{I}_i}\left|\overline{L(m)}_{i,s}\right|\left|\overline{I}_i\right||\mathcal{F}_{t^n_{i-1}}\right]
\lesssim_m\left(k_n\bar{r}_n\right)^{3/2}.
\end{align*}
Consequently, \eqref{aimIII-2} holds true for $l=2$. Finally, by \ref{hypo:SA2}, \eqref{SA4} and \eqref{moment:jump} we have
\begin{align*}
E\left[\left|\mathbf{A}^{n,3}_i\right||\mathcal{F}_{t^n_{i-1}}\right]
&\lesssim_mE\left[\sup_{s\in\overline{I}_i}\left|\overline{X'(m)}_{i,s}\right||\overline{I}_i||\mathcal{F}_{t^n_{i-1}}\right]
\lesssim\left(k_n\bar{r}_n\right)^2,
\end{align*}
hence \eqref{aimIII-2} holds true for $l=3$. This completes the proof.
\end{proof}

\begin{proof}[Proof of \eqref{aim5}]
Define the processes $B'(m)$ and $Z'(m)$ by $B'(m)_t=\int_0^t\int_{A_m}\delta(s,z)ds\lambda(dz)$ and $Z'(m)_t=X'(m)_t-B'(m)_t=\int_0^t\int_{A_m}\delta(s,z)(\mu-\nu)(ds,dz)$.  Since we have
\begin{align*}
\left|\mathbb{V}_n(m)\right|
&\leq3\frac{\Delta_n^{-1/4}}{\psi_3k_n}\sum_{i=0}^{N^n_T-k_n+1}\left\{\left|\overline{Z(m)}_i\left(\overline{B'(m)}_i\right)^2\right|+\left|\overline{Z(m)}_i\left(\overline{Z'(m)}_i\right)^2\right|\right\}\\
&=:\mathbb{V}_n^{(1)}(m)+\mathbb{V}_n^{(2)}(m),
\end{align*}
it suffices to prove
\begin{equation}\label{aim5aim1}
\lim_{m\to\infty}\limsup_{n\to\infty}P_n\left(\mathbb{V}_n^{(l)}(m)>\eta\right)=0
\end{equation}
for $l=1,2$.

We have
\begin{align*}
E\left[\left|\overline{Z(m)}_i\left(\overline{B'(m)}_i\right)^2\right||\mathcal{F}_{t^n_{i-1}}\right]
\lesssim_m (k_n\bar{r}_n)^2E\left[\left|\overline{Z(m)}_i\right||\mathcal{F}_{t^n_{i-1}}\right]
\lesssim_m(k_n\bar{r}_n)^{5/2}
\end{align*}
by \ref{hypo:SA2}, \eqref{SA4} and \eqref{moment:Z}. Therefore, \eqref{aim5aim1} holds true for $l=1$ by the Lenglart inequality, \eqref{C3} and the fact that $\xi>4/5$.

Now we prove \eqref{aim5aim1} for $l=2$. We start with introducing a further localization procedure for the observation times. For each $K\in\mathbb{N}$, we define the sequence $(t^n_{i}(K))_{i=-1}^\infty$ inductively by $t^n_{-1}(K)=0$ and
\begin{equation*}
t^n_{i}(K)=
\left\{\begin{array}{ll}
t^n_i, & \textrm{if $\max_{j=1,\dots,i}E\left[\Delta_n^{-1}|I_j|\big|\mathcal{F}_{t^n_{j-1}}\right]\leq K$ and $\Delta_nN^n_{t^n_{i-1}}\leq K$},\\
t^n_{i-1}(K)+\Delta_n, & \textrm{otherwise}.
\end{array}\right.
\end{equation*}
By construction $t^n_{i}(K)$ is an $(\mathcal{F}_t)$-stopping time for every $i$. Moreover, by \ref{hypo:A1} and \eqref{C3} we have $\sup_nP(t^n_i\neq t^n_i(K)\text{ for some }i\in\{0,1,\dots,N^n_T\})\to0$ as $K\to\infty$. Consequently, it suffices to show that
\begin{equation}\label{aim5aim2}
\lim_{m\to\infty}\limsup_{n\to\infty}P_n\left(\mathbb{V}_n^{(2)}(m)>\eta,t^n_i=t^n_i(K) \text{ for any }i\in\{0,1,\dots,N^n_T\}\right)=0
\end{equation}
for any fixed $K\in\mathbb{N}$.

Set $\widetilde{I}_i=[t^n_{i-1}(K),t^n_{i}(K))$ and define the process $\widetilde{g}^n_i$ by $\widetilde{g}^n_i(s)=\sum_{p=1}^{k_n-1}g^n_p1_{\widetilde{I}_{i+p}}(s)$. For any semimartingale $V$, we define the process $\widetilde{V}_{i,t}$ by $\widetilde{V}_{i,t}=\int_0^t\widetilde{g}^n_i(s-)dV_s$. Then, to prove \eqref{aim5aim2} it is enough to show that
\begin{equation}\label{aim5aim3}
\lim_m\limsup_nE\left[\frac{\Delta_n^{-1/4}}{k_n}\sum_{i=0}^{\widetilde{N}^n_T+1}\left|\widetilde{Z(m)}_{i,t^n_{i+k_n-1}(K)}\left(\widetilde{Z'(m)}_{i,t^n_{i+k_n-1}(K)}\right)^2\right|\right]
=0,
\end{equation}
where $\widetilde{N}^n_T=\max\{i:t^n_i(K)\leq T\}$. Note that $\widetilde{N}^n_T\leq (K+T)\Delta_n^{-1}$ by construction. 

Set $\widetilde{I}^+_i=[t^n_{i-1}(K),t^n_{i+k_n-1}(K))$. 
To prove \eqref{aim5aim3}, we consider the following decomposition, which is obtained by applying integration by parts repeatedly (note that $[Z(m),Z'(m)]\equiv0$ by construction):
\begin{align*}
\widetilde{Z(m)}_{i,t^n_{i+k_n-1}(K)}\left(\widetilde{Z'(m)}_{i,t^n_{i+k_n-1}(K)}\right)^2
&=\int_{\widetilde{I}_i^+}\left(\widetilde{Z'(m)}_{i,s-}\right)^2d\widetilde{Z(m)}_{i,s}
+2\int_{\widetilde{I}_i^+}\widetilde{Z(m)}_{i,s-}\widetilde{Z'(m)}_{i,s-}d\widetilde{Z'(m)}_{i,s}\\
&+\int_{\widetilde{I}_i^+}\widetilde{Z(m)}_{i,s-}d[\widetilde{Z'(m)}_{i,\cdot}]_s\\
&=:\mathbf{A}(m)^{n,1}_i+2\mathbf{A}(m)^{n,2}_i+\mathbf{A}(m)^{n,3}_i.
\end{align*}
Then, it is enough to show that
\begin{equation}\label{aimV}
\lim_m\limsup_nE\left[\frac{\Delta_n^{-1/4}}{k_n}\sum_{i=0}^{\widetilde{N}^n_T+1}\left|\mathbf{A}(m)^{n,l}_i\right|\right]
=0
\end{equation}
for every $l=1,2,3$. First, we consider the case $l=1$. Lemma \ref{lenglart2} and \ref{hypo:SA2} yield
\begin{align*}
E\left[\left|\mathbf{A}(m)^{n,1}_i\right||\mathcal{F}_{t^n_{i-1}(K)}\right]
&\lesssim \sqrt{\overline{\gamma}_m}E\left[\sup_{s\in\widetilde{I}^+_i}\left(\widetilde{Z'(m)}_{i,s}\right)^2\sqrt{|\widetilde{I}_i^+|}|\mathcal{F}_{t^n_{i-1}(K)}\right]\\
&\lesssim\sqrt{\overline{\gamma}_m}E\left[\sup_{s\in\widetilde{I}^+_i}\left(\widetilde{Z'(m)}_{i,s}\right)^2\sqrt{\left|\sum_{p=1}^{k_n-1}\left(|\widetilde{I}_{i+p}|-E\left[|\widetilde{I}_{i+p}|\big|\mathcal{F}_{T^K_{i+p-1}}\right]\right)\right|}|\mathcal{F}_{t^n_{i-1}(K)}\right]\\
&\qquad+\sqrt{\overline{\gamma}_m}E\left[\sup_{s\in\widetilde{I}^+_i}\left(\widetilde{Z'(m)}_{i,s}\right)^2\sqrt{\sum_{p=1}^{k_n-1}E\left[|\widetilde{I}_{i+p}|\big|\mathcal{F}_{T^K_{i+p-1}}\right]}|\mathcal{F}_{t^n_{i-1}(K)}\right]\\
&=:\sqrt{\overline{\gamma}_m}\left(\mathbf{B}(m)^{n,1}_i+\mathbf{B}(m)^{n,2}_i\right).
\end{align*}
It suffices to prove
\begin{equation}\label{aimV-1}
\lim_m\limsup_nE\left[\frac{\Delta_n^{-1/4}}{k_n}\sum_{i=0}^{\widetilde{N}^n_T+1}\sqrt{\overline{\gamma}_m}\mathbf{B}(m)^{n,j}_i\right]=0
\end{equation}
for $j=1,2$. By the H\"older and BDG inequalities and \eqref{SA4}, we have
\begin{align*}
\mathbf{B}(m)^{n,1}_i
&\lesssim(k_n\bar{r}_n^2)^{1/4}\left\{E\left[\sup_{s\in\widetilde{I}^+_i}\left(\widetilde{Z'(m)}_{i,s}\right)^{2p}|\mathcal{F}_{t^n_{i-1}(K)}\right]\right\}^{1/p}
\end{align*}
for any $p\in(1,2]$. Therefore, the Novikov inequality (Theorem 1 of \cite{Novikov1975}) implies that
\begin{align*}
\mathbf{B}(m)^{n,1}_i
&\lesssim(k_n\bar{r}_n^2)^{1/4}(k_n\bar{r}_n)^{1/p}
\end{align*}
for any $p\in(1,2]$. Now, we can take $p\in(1,\frac{\xi-\frac{1}{2}}{\frac{7}{8}-\frac{\xi}{2}})$ because $\xi>\frac{11}{12}$, hence the above inequality yields \eqref{aimV-1} for $j=1$. On the other hand, the construction of $(t^n_i(K))$ and the Doob inequality imply that
\begin{align*}
\mathbf{B}(m)^{n,2}_i
&\lesssim\sqrt{k_n\Delta_n}E\left[\sup_{s\in\widetilde{I}^+_i}\left(\widetilde{Z'(m)}_{i,s}\right)^{2}|\mathcal{F}_{t^n_{i-1}(K)}\right]
\lesssim\sqrt{k_n\Delta_n}E\left[\left|\widetilde{I}^+_i\right||\mathcal{F}_{t^n_{i-1}(K)}\right]
\lesssim\left(k_n\Delta_n\right)^{3/2},
\end{align*}
hence \eqref{aimV-1} also holds true for $j=2$.

Next consider the case $l=2$. 
Lemma \ref{lenglart2}, \ref{hypo:SA2} and the Schwarz inequality yield
\begin{align*}
E\left[\left|\mathbf{A}(m)^{n,2}_i\right||\mathcal{F}_{t^n_{i-1}(K)}\right]
&\lesssim E\left[\sup_{s\in\widetilde{I}^+_i}\left|\widetilde{Z(m)}_{i,s}\widetilde{Z'(m)}_{i,s}\right|\sqrt{\left|\widetilde{I}_i^+\right|}|\mathcal{F}_{t^n_{i-1}(K)}\right]\\
&\lesssim \sqrt{E\left[\sup_{s\in\widetilde{I}^+_i}\left|\widetilde{Z(m)}_{i,s}\widetilde{Z'(m)}_{i,s}\right|^2|\mathcal{F}_{t^n_{i-1}(K)}\right]E\left[\left|\widetilde{I}_i^+\right||\mathcal{F}_{t^n_{i-1}(K)}\right]}.
\end{align*}
Noting $[Z(m),Z'(m)]\equiv0$ by construction, we obtain the following identity for $s\in\widetilde{I}^+_i$ by applying integration by parts:
\begin{align*}
\widetilde{Z(m)}_{i,s}\widetilde{Z'(m)}_{i,s}
=\int_{t^n_{i-1}(K)}^s\widetilde{Z(m)}_{i,u-}d\widetilde{Z'(m)}_{i,u}+\int_{t^n_{i-1}(K)}^s\widetilde{Z'(m)}_{i,u-}d\widetilde{Z(m)}_{i,u}.
\end{align*} 
Therefore, the Doob inequality yields
\begin{align*}
E\left[\sup_{s\in\widetilde{I}^+_i}\left|\widetilde{Z(m)}_{i,s}\widetilde{Z'(m)}_{i,s}\right|^2|\mathcal{F}_{t^n_{i-1}(K)}\right]
&\lesssim E\left[\left(\sup_{s\in\widetilde{I}^+_i}\widetilde{Z(m)}_{i,s}^2+\overline{\gamma}_m\sup_{s\in\widetilde{I}^+_i}\widetilde{Z'(m)}_{i,s}^2\right)\left|\widetilde{I}_i^+\right||\mathcal{F}_{t^n_{i-1}(K)}\right].
\end{align*}
Hence, by an analogous argument to the proof of the case $l=1$ we obtain
\[
E\left[\sup_{s\in\widetilde{I}^+_i}\left|\widetilde{Z(m)}_{i,s}\widetilde{Z'(m)}_{i,s}\right|^2|\mathcal{F}_{t^n_{i-1}(K)}\right]\lesssim\overline{\gamma}_m(k_n\Delta_n)^2.
\]
Consequently, we conclude that
\[
E\left[\left|\mathbf{A}(m)^{n,2}_i\right||\mathcal{F}_{t^n_{i-1}(K)}\right]
\lesssim\sqrt{\overline{\gamma}_m}(k_n\Delta_n)^{3/2},
\]
and thus we obtain \eqref{aimV} for $l=2$. 

Finally consider the case $l=3$. Since \ref{hypo:SA2} yields
\begin{align*}
E\left[\left|\mathbf{A}(m)^{n,3}_i\right||\mathcal{F}_{t^n_{i-1}(K)}\right]
&\lesssim E\left[\sup_{s\in\widetilde{I}^+_{i}}\left|\widetilde{Z(m)}_{i,s}\right|\left|\widetilde{I}^+_{i}\right||\mathcal{F}_{t^n_{i-1}(K)}\right],
\end{align*}
we can again apply an analogous argument to the proof of the case $l=1$, and thus \eqref{aimV} holds true for $l=3$. This completes the proof.
\end{proof}

\begin{proof}[Proof of \eqref{aim2}]
We decompose the target quantity as
\begin{align*}
&\mathbb{II}_n(m)\\
&=\frac{\Delta_n^{-1/4}}{\psi_3k_n}\sum_{i=0}^{N^n_T-k_n+1}\left(\overline{C(m)}_i\right)^3+\frac{3\Delta_n^{-1/4}}{\psi_3k_n}\sum_{i=0}^{N^n_T-k_n+1}\left(\overline{C(m)}_i\right)^2\overline{Z(m)}_i+\frac{3\Delta_n^{-1/4}}{\psi_3k_n}\sum_{i=0}^{N^n_T-k_n+1}\overline{C(m)}_i\left(\overline{Z(m)}_i\right)^2\\
&\qquad+\Delta_n^{-1/4}\left\{\frac{1}{\psi_3k_n}\sum_{i=0}^{N^n_T-k_n+1}\left(\overline{Z(m)}_i\right)^3-\sum_{0\leq s\leq T}(\Delta X(m)_s)^3\right\}\\
&=:\mathbb{II}_n^{(1)}(m)+\mathbb{II}_n^{(2)}(m)+\mathbb{II}_n^{(3)}(m)+\mathbb{II}_n^{(4)}(m).
\end{align*}
It suffices to prove
\begin{equation}\label{aimII}
\lim_{m\to\infty}\limsup_{n\to\infty}P_n\left(\left|\mathbb{II}_n^{(l)}(m)\right|>\eta\right)=0
\end{equation}
for every $l=1,2,3,4$.

Since Proposition 4.1 of \cite{Koike2015} yields $\mathbb{II}_n^{(1)}(m)\to^P0$ as $n\to\infty$ for every $m$, \eqref{aimII} holds true for $l=1$. Moreover, we can prove \eqref{aimII} for $l=2,3$ analogously to the proof of \eqref{aim5}. So it remains to prove \eqref{aimII} for $l=4$. Applying integration by parts repeatedly, we can decompose the target quantity as
\begin{align*}
\mathbb{II}_n^{(4)}(m)
&=\frac{3\Delta_n^{-1/4}}{\psi_3k_n}\sum_{i=0}^{N^n_T-k_n+1}\int_{\overline{I}_i}\left(\overline{Z(m)}_{i,s-}\right)^2d\overline{Z(m)}_{i,s}
+\frac{3\Delta_n^{-1/4}}{\psi_3k_n}\sum_{i=0}^{N^n_T-k_n+1}\int_{\overline{I}_i}\overline{Z(m)}_{i,s-}d[\overline{Z(m)}_{i,\cdot}]_s\\
&\qquad+\Delta_n^{-1/4}\left\{\frac{1}{\psi_3k_n}\sum_{i=0}^{N^n_T-k_n+1}\left[\overline{Z(m)}_{i,\cdot},[\overline{Z(m)}_{i,\cdot}]\right]_{t^n_{i+k_n-1}}-\sum_{0\leq s\leq T}(\Delta X(m)_s)^3\right\}\\
&=:\mathbb{A}^{(1)}_n(m)+\mathbb{A}^{(2)}_n(m)+\mathbb{A}^{(3)}_n(m),
\end{align*}
hence it is enough to prove
\begin{equation}\label{aimII-4}
\lim_{m\to\infty}\limsup_{n\to\infty}P_n\left(\left|\mathbb{A}_n^{(l)}(m)\right|>\eta\right)=0
\end{equation}
for every $l=1,2,3$.  For $l=1,2$, \eqref{aimII-4} can be shown analogously to the proof of \eqref{aim5aim1} for $l=2$. On the other hand, since we have $[\overline{Z(m)}_{i,\cdot}]_s=\sum_{p=1}^{k_n-1}(g^n_p)^2\sum_{t^n_{i+p-1}<u\leq s}(\Delta X(m)_u)^2$ for $s\in\overline{I}_i$ and $\Delta Z(m)=\Delta X(m)$, we obtain
\begin{align*}
\mathbb{A}^{(3)}_n(m)
&=\Delta_n^{-1/4}\left\{\frac{1}{\psi_3k_n}\sum_{p=1}^{N^n_T}\left(\sum_{i=(p-k_n+1)_+}^{(p-1)\wedge(N^n_T-k_n+1)}(g^n_{p-i})^3\right)\sum_{t^n_{p-1}<s\leq t^n_{p}}\left(\Delta X(m)_s\right)^3-\sum_{0\leq s\leq T}(\Delta X(m)_s)^3\right\}.
\end{align*}
Now, since we have $\sum_{(t-h)_+<s\leq t}\left|\Delta X(m)_s\right|^3=O_p(h)$ as $h\downarrow0$ by \ref{hypo:SA2}, we can deduce that $\mathbb{A}^{(3)}_n(m)\to0$ as $n\to\infty$ for every $m$, so \eqref{aimII-4} holds true for $l=3$.
\end{proof}

\section*{References}
\bibliographystyle{model2-names}

\end{document}